\documentclass[11pt]{article}
\usepackage{amsfonts, amsmath, amssymb, amscd, amsthm, color, graphicx, mathrsfs, mathabx, wasysym, setspace, mdwlist, calc,float}
\usepackage{setspace}
\usepackage{hyperref}
\usepackage{tikz-cd} 
\usepackage{hyperref}
\usepackage{tocloft}
\usepackage{xcolor}

\usepackage{hyperref}
\hypersetup{linktocpage}

\hypersetup{colorlinks,
    linkcolor={red!50!black},
    citecolor={blue!80!black},
    urlcolor={blue!80!black}}
\usepackage{float}

\usepackage{comment}

\newcommand{\NN}{\mathbb{N}}
\newcommand{\ZZ}{\mathbb{Z}}

\newcommand{\RR}{\mathbb{R}}

\newcommand{\NNo}{\mathbb{N}\cup\{0\}}

\newcommand{\act}{\curvearrowright}

\newcommand{\h}{\mathfrak{h}}

\newcommand{\G}{\mathcal{G}}
\newcommand{\R}{\mathfrak{R}}
\renewcommand{\H}{\mathcal{H}}

\newtheorem{thm}{Theorem}[section]
\newtheorem{cor}[thm]{Corollary}
\newtheorem{lem}[thm]{Lemma}
\newtheorem{prop}[thm]{Proposition}

\theoremstyle{definition}
\newtheorem{defn}[thm]{Definition}

\theoremstyle{remark}
\newtheorem{rem}[thm]{Remark}

\newtheorem{que}[thm]{Question}

\newcommand{\I}{\mathcal{I}}

\newcommand{\la}{\langle}
\newcommand{\ra}{\rangle}
\newcommand{\lab}{\mathbf{Lab}}

\newcommand{\supp}{\mathrm{supp}}
\newcommand{\aut}{\mathrm{Aut}}
\newcommand{\link}{\mathrm{link}}
\newcommand{\lk}{\mathrm{Lk}}
\newcommand{\deck}{\mathrm{Deck}}

\begin{document}

\title{Hyperfiniteness of the boundary action of virtually special groups}

\author{Koichi Oyakawa}
\date{}

\maketitle

\vspace{-10mm}

\begin{abstract}
    We prove that for any countable group acting virtually specially on a CAT(0) cube complex, the orbit equivalence relation induced by its action on the Roller boundary is hyperfinite. This can be considered as a generalization of hyperfiniteness of the boundary action of cubulated hyperbolic groups by Huang-Sabok-Shinko.
\end{abstract}

\section{Introduction}

Borel complexity measures how complicated it is to describe equivalence relations on standard Borel spaces. It provides a formulation of classification problems, and the study of Borel complexity has long been an active topic in descriptive set theory. Despite its long history, there are still many fundamental long standing open problems. One of them is the following (see \cite[Section 16.4]{Kec25}).
\begin{que}\label{que:hyperfinite}
    Is every measure-hyperfinite countable Borel equivalence relation hyperfinite?
\end{que}
Roughly speaking, a Borel equivalence relation $E$ on a standard Borel space $X$ is called hyperfinite if $E$ can be approximated by Borel equivalence relations with finite classes and called measure-hyperfinite if $E$ is $\mu$-hyperfinite (i.e. $E$ becomes hyperfinite after removing some $\mu$-null set of $X$) for every Borel probability measure $\mu$ on $X$ (see Section \ref{sec:Descriptive set theory}). Measure-hyperfiniteness is weaker than hyperfiniteness by default and Question \ref{que:hyperfinite} asks whether these two notions are equivalent.

One of the main sources of measure-hyperfinite Borel equivalence relations is the orbit equivalence relations of topologically amenable actions by Connes-Feldman-Weiss Theorem (see \cite{CFW81} and \cite[Appendix A]{FKSV25}). Topological amenability of group actions is a far-reaching generalization of amenability of groups and has been extensively studied with its connection to group theory, geometry, ergodic theory, and operator algebras. It is known that various group actions on nonpositively curved spaces induce a topologically amenable action on a natural boundary of the space, even if the group itself is not amenable (see \cite{Ada94,Kai04, Oza06,Kid08,Ham09,Lec10,GN11,NS13,HH21,BGC22}). Therefore, it is interesting to ask whether such natural actions appearing in geometry induce hyperfinite orbit equivalence relations in the quest of Question \ref{que:hyperfinite}. This can also be considered as the investigation of a group action version of Weiss Conjecture.

A classical result in this direction is that the natural action of the free group $F_2$ on its Gromov boundary induces the hyperfinite orbit equivalence relation, which follows from the hyperfiniteness of tail equivalence relations by Dougherty-Jackson-Kechris (see \cite{DJK94}). A breakthrough was achieved in \cite{HSS20} by Huang-Sabok-Shinko, where they generalized this result to every cubulated hyperbolic group. Later, this was further generalized to every hyperbolic group by Marquis-Sabok in \cite{MS20} and led to extensive research on hyperfiniteness of the orbit equivalence relation on the Gromov boundary induced by natural group actions on hyperbolic spaces (see \cite{Mar19,PS21,Kar22,Oya24,KEOSS24,NV25,KOO26}).

The purpose of this paper is to open up another generalization of \cite{HSS20}, which is for group actions on CAT(0) cube complexes. Theorem \ref{thm:main} below is our main theorem (see Section \ref{sec:Nonpositively curved cube complexes} for relevant notions). We actually prove a more general result in Theorem \ref{thm:hyperfinite of virt special groups} without cocompactness. In Section \ref{sec:Cubulated hyperbolic group}, we verify that this is indeed a generalization of \cite{HSS20} (see Corollary \ref{cor:hyperfinite Roller bdry of cubulated hyp gp}). The point here is that \cite{HSS20} considers actions on the Gromov boundary, while Theorem \ref{thm:main} considers actions on the Roller boundary.
\begin{thm}\label{thm:main}
    Let $X$ be a CAT(0) cube complex and $G$ be a countable group acting virtually cocompact specially on $X$. Then, the orbit equivalence relation $E_G^{\partial_\R X}$ induced by the action of $G$ on the Roller boundary $\partial_\R X$ of $X$ is hyperfinite.
\end{thm}
Special cube complexes are a very important class of cube complexes introduced by Haglund-Wise in \cite{HW08}. It is a higher-dimensional generalization of graphs and its fundamental group has nice separability properties such as residual finiteness. This notion was a crucial piece of the solution of Virtual Haken Conjecture and Virtual Fibering Conjecture by Agol in \cite{Ian13}, where Agol proved that every cubulated hyperbolic group is virtually special. Here, it is worth mentioning that Guentner-Niblo showed in \cite{GN11} that if a countable group acts on a finite dimensional CAT(0) cube complex, then the induced action on the Roller compactification is topologically amenable if and only if every vertex stabilizer is amenable (see \cite[Section 4]{GN11} and \cite{NS13}).

As for the proof of Theorem \ref{thm:main}, the crucial difference from all previous works on the Borel complexity of boundary actions is that geodesic rays converging to a common boundary point do not satisfy the fellow traveling property. This makes it impossible to apply previous methods that construct a Borel reduction from the orbit equivalence relation on the boundary to a tail equivalence relation by using edge labels. This difference can also be seen from the fact that virtually special groups can be amenable groups, for which hyperfiniteness of Borel actions is usually proved by a different approach. This lack of both hyperbolicity and amenability is the source of difficulty in proving Theorem \ref{thm:main}. A novel idea to tackle this difficulty is to use coloring of hyperplanes in the universal cover of a special cube complex, which reduces to the study of right angled Artin groups in Section \ref{sec:Hyperfiniteness in the case of right angled Artin groups}, and to construct a Borel reduction to product of tail equivalence relations.

Theorem \ref{thm:main} (or Theorem \ref{thm:hyperfinite of virt special groups}) also implies the following new results on the action on the Gromov boundary of hyperbolic graphs associated to CAT(0) cube complexes and also the boundary action of Coxeter groups.

\begin{thm}\label{thm:intro contact graph}
    Let $X$ be a CAT(0) cube complex and $G$ be a countable group acting virtually cocompact specially on $X$. Let $\mathscr{C} X$ be the contact graph of $X$. Then, the orbit equivalence relation $E_G^{\partial_\infty \mathscr{C} X}$ induced by the action of $G$ on the Gromov boundary $\partial_\infty \mathscr{C} X$ of $\mathscr{C} X$ is hyperfinite.
\end{thm}

\begin{thm}\label{thm:intro RAAG}
    Let $\Gamma$ be a finite simple graph, $A(\Gamma)$ be the right angled Artin group of $\Gamma$, and $\Gamma^e$ be the extension graph of $\Gamma$. Then, the orbit equivalence relation $E_{A(\Gamma)}^{\partial_\infty\Gamma^e}$ induced by the action of $A(\Gamma)$ on the Gromov boundary $\partial_\infty \Gamma^e$ of $\Gamma^e$ is hyperfinite.
\end{thm}

\begin{thm}\label{thm:intro Coxeter groups}
    Let $G$ be a finitely generated Coxeter group and let $C$ be the Niblo-Reeves CAT(0) cube complex on which $G$ acts properly. Then, the orbit equivalence relation $E_G^{\partial_\R C}$ induced by the action of $G$ on the Roller boundary $\partial_\R C$ of $C$ is hyperfinite.
\end{thm}

This paper is organized as follows. In Section \ref{Preliminaries}, we discuss the necessary definitions and known results about descriptive set theory, cube complexes, and hyperbolic spaces. In Section \ref{sec:Hyperfiniteness in the case of right angled Artin groups}, we prove hyperfiniteness for a natural action of right angled Artin groups. In Section \ref{sec:Proof of main theorem}, we prove Theorem \ref{thm:main} using results in Section \ref{sec:Hyperfiniteness in the case of right angled Artin groups}. In Section \ref{sec:Cubulated hyperbolic group}, we verify that Theorem \ref{thm:main} is a generalization of \cite{HSS20}.

\section{Preliminaries}\label{Preliminaries}

In this section, we discuss the necessary definitions and known results. We start by introducing basic notation related to graphs, metric spaces, and groups.

\begin{defn}
     Let $X$ be a graph. For a path $\gamma=(v_0,\ldots,v_n)$ in $X$ without self-intersection, where $v_i$'s are vertices, we denote by $\gamma_-$ the initial vertex of $\gamma$ and by $\gamma_+$ the terminal vertex of $\gamma$ (i.e. $\gamma_-=v_0$ and $\gamma_+=v_n$). We denote by $\gamma_{[v_i,v_j]}$ the subpath of $\gamma$ from $v_i$ to $v_j$ with $i<j$. For a geodesic ray $\gamma = (v_0,v_1,\ldots)$ in $\Gamma$, we denote by $\gamma_{[v_i,\infty)}$ the subray from $v_i$ i.e. $\gamma_{[v_i,\infty)}=(v_i,v_{i+1},\ldots)$. A graph is called \emph{simple} if it has no loops or multiple edges.
\end{defn}

\begin{defn}
    Let $(X,d)$ be a metric space. For $A,B\subset X$, we define $d(A,B)=\inf_{x \in A,\,y\in B}d(x,y)$.
\end{defn}

\begin{defn}
    Let $G$ be a group. For $S \subset G$, we denote by $\la S \ra$ the subgroup of $G$ generated by $S$. Given a generating set $S$ of $G$ (i.e. $G=\la S \ra$), we define the \emph{Cayley graph of $G$ with respect to $S$}, denoted by $Cay(G,S)$, to be the simple labeled graph whose vertex set is $G$ and $g,h \in G$ are adjacent if and only if $h=gs$ with some $s \in S\cup S^{-1}$, where the edge from $g$ to $h$ is labeled by $s$. Given a path $\gamma$ in $Cay(G,S)$, we denote by $\lab(\gamma)$ the word in $S\cup S^{-1}$ obtained by concatenating the edge labels of $\gamma$.
\end{defn}

\begin{defn}
    An action of a group $G$ on a set $X$ is called \emph{free} if for any $x \in X$, $\{g \in G \mid gx = x\} = \{1\}$. Given sets $X,Y$, a map $f \colon X \to Y$ is called \emph{countable-to-1} (resp. \emph{finite-to-1}) if for any $y \in Y$, $f^{-1}(y)$ is countable (resp. finite). For a subset $A \subset X\times Y$ and $x \in X$, we define $A_x \subset Y$ by $A_x=\{y \in Y \mid (x,y) \in A\}$.
\end{defn}

\subsection{Descriptive set theory}\label{sec:Descriptive set theory}

In this section, we review descriptive set theory. See \cite{Kec95,Anu22,Kec25} for further details on descriptive set theory and Borel equivalence relations.

\begin{defn}\label{def:eq ref}
    Let $X$ be a set. An \emph{equivalence relation on} $X$ is a reflexive, symmetric, and transitive subset of $X^2$. Let $E$ be an equivalence relation on $X$. For $x,y\in X$, we denote $x\, E \,y$ when $(x,y)\in E$. For $A \subset X$, we define an equivalence relation $E|_A$ on $A$ by $E|_A=E \cap(A\times A)$. An equivalence relation $E$ is called \emph{countable} (resp. \emph{finite}) if for every $x \in X$, the set $\{y \in X \mid x\, E \, y\}$ is countable (resp. finite).
\end{defn}

\begin{defn}
    A \emph{Polish space} is a topological space that is separable and completely metrizable. A measurable space $(X,\mathcal{B})$ is called a \emph{standard Borel space} if there exists a Polish topology $\mathcal{O}$ on $X$ such that $\mathcal{B}$ is the smallest $\sigma$-algebra containing $\mathcal{O}$. 
\end{defn}

\begin{rem}\label{rem:union of sbs}
    We will often shorten ``Borel measurable" to ``Borel".
\end{rem}

\begin{defn}\label{def:Borel eq rel}
    Let $X$ be a standard Borel space. An equivalence relation $E$ on $X$ is called \emph{Borel} if $E\subset X^2$ is Borel in $X^2$. A Borel equivalence relation $E$ on $X$ is called 
    \begin{itemize}
        \item[-]
        \emph{smooth} if there exists a standard Borel space $Y$ and a Borel map $f \colon X \to Y$ such that for any $x,y \in X$, $x\, E \, y \iff f(x) = f(y)$.
        \item[-]
        \emph{hypersmooth} if there exist smooth Borel equivalence relations $\{E_n\}_{n \in \NN}$ on $X$ such that $E = \bigcup_{n \in \NN}E_n$ and $\forall\,n \in \NN,\,E_n \subset E_{n+1}$.
        \item[-]
        \emph{hyperfinite} if there exist finite Borel equivalence relations $\{E_n\}_{n \in \NN}$ on $X$ such that $E = \bigcup_{n \in \NN}E_n$ and $\forall\,n \in \NN,\,E_n \subset E_{n+1}$.
    \end{itemize} 
\end{defn}

\begin{rem}
    We will often abbreviate ``countable Borel equivalence relation" to CBER.
\end{rem}

Definition \ref{def:orbit equivalence} and Definition \ref{def:tail equivalence} provide two important examples of CBERs in this paper.

\begin{defn}\label{def:orbit equivalence}
     Suppose that a group $G$ acts on a set $X$. The equivalence relation $E_G^X$ on $X$ is defined as follows; for $x,y\in X$,
     \[
        (x,y) \in E_G^X \iff \exists\, g\in G,\, gx=y.
     \]
     $E_G^X$ is called the \emph{orbit equivalence relation} on $X$.
\end{defn}

\begin{rem}\label{rem:orbit equivalence is cber}
    If a countable group $G$ acts on a standard Borel space $X$ as Borel isomorphism, then $E_G^X$ is a CBER.
\end{rem}

\begin{defn}\label{def:tail equivalence}
        Let $X$ be a set. The equivalence relation $E_t(X)$ on $X^\NN$ is defined as follows; for $w_0=(s_1,s_2,\ldots), w_1=(t_1,t_2,\ldots) \in X^\NN$,
    \[
    (w_0,w_1)\in E_t(X) \iff \exists\, n, \exists\, m \in\NN\cup\{0\},\, \forall\, i\in\NN,\, s_{n+i}=t_{m+i}.
    \]
    $E_t(X)$ is called the \emph{tail equivalence relation} on $X^\NN$.
\end{defn}

We list some facts needed for the proof of Theorem \ref{thm:main}. Theorem \ref{thm:DJK} below is from \cite[Theorem 8.1]{DJK94}. Recall that any countable set $\Omega$ with the discrete topology is a Polish space. Also, for any Polish space $X$, the set $X^\NN$ endowed with the product topology is a Polish space.

\begin{thm}[cf. {\cite[Theorem 8.1]{DJK94}}]\label{thm:DJK}
For any standard Borel space $X$, the tail equivalence relation $E_t(X)$ on $X^\NN$ is a hypersmooth Borel equivalence relation. Moreover, for any countable set $\Omega$, the tail equivalence relation $E_t(\Omega)$ on $\Omega^\NN$ is a hyperfinite CBER.
\end{thm}

\begin{prop}[{\cite[Proposition 1.3.(vii)]{JKL02}}]\label{prop:JKL}
    Let $X$ be a standard Borel space and $E,F$ be countable Borel equivalence relations on $X$. If $E\subset F$, $E$ is hyperfinite, and every $F$-equivalence class contains only finitely many $E$-classes, then $F$ is hyperfinite.
\end{prop}

Theorem \ref{thm:Lusin-Novikov} below (see \cite[Theorem 18.10]{Kec95}) is used to show Lemma \ref{lem:hyperfininte when finite to 1}.

\begin{thm}[Lusin-Novikov]\label{thm:Lusin-Novikov}
    Let $X,Y$ be standard Borel spaces and let $P \subset X \times Y$ be Borel. If $P_x$ is countable for any $x \in X$, then there exists a Borel set $Q \subset P$ such that for any $x \in X$, $P_x \neq \emptyset \iff |Q_x|=1$.
\end{thm}

Theorem \ref{thm:Arsenin-Kunugui} below is used in this paper for the case when the section $A_x$ is a countable set, which is $K_\sigma$ (see \cite[Theorem 18.18, Theorem 35.43]{Kec95}, \cite[Corollary 13.7]{Anu22}).

\begin{thm}[Arsenin-Kunugui]\label{thm:Arsenin-Kunugui}
    Let $Y$ be a Polish space, $X$ a standard Borel space, and $A \subset X \times Y$ Borel such that $A_x$ is $K_\sigma$ (i.e. countable union of compact sets) in $Y$ for any $x \in X$. Then, the set $\mathrm{proj}_X(A)\,(=\{x \in X\mid \exists\,y \in Y,\,(x,y) \in A\})$ is Borel.
\end{thm}

\subsection{Nonpositively curved cube complexes}\label{sec:Nonpositively curved cube complexes}

In this section, we review nonpositively curved cube complexes. We follow \cite{Wis12}. See \cite{Wis12,BSV14} for further details on cube complexes.

\begin{defn}
    For $n \in \NNo$, a $n$-\emph{cube} is $[0,1]^n$. (A $0$-cube is a singleton by definition.) Given a $n$-cube $\sigma$, a \emph{face} of $\sigma$ is a cube obtained by restricting some coordinates to $0$ or $1$, and a \emph{midcube} of $\sigma$ is a $(n-1)$-cube obtained by restricting exactly one coordinate to $\frac{1}{2}$. A \emph{cube complex} is a CW complex such that each cell is a $n$–cube for some $n$ and each cell is attached using an isometry of some face. Given a cube complex $X$, we denote by $X^{(n)}$ the set of all $n$-cubes in $X$ for each $n \in \NNo$ and define the $n$-\emph{skeleton} of $X$ to be the subcomplex formed by $\bigcup_{i=0}^n X^{(i)}$. The \emph{link} of a $0$-cube $v \in X^{(0)}$, denoted by $\link_X(v)$, is the simplex-complex whose $n$-simplices are corners of $(n+1)$-cubes adjacent with $v$. The \emph{dimension} of a cube complex $X$ is the least $N \in \NNo\cup\{\infty\}$ such that $X = \bigcup_{i=0}^N X^{(i)}$. We say $X$ is \emph{finite dimensional} if $N<\infty$. A cube complex $X$ is called \emph{nonpositively curved (NPC)} if the link of every $0$-cube is a flag complex. (Here, a \emph{flag complex} is a simplicial complex such that for every $n \in \NN$, $n+1$ vertices span a $n$-simplex if and only if they are pairwise adjacent). 
\end{defn}

\begin{defn}
    Let $X$ be a NPC cube complex. Define an equivalence relation $\sim$ on $X^{(1)}$ as follows; for $1$-cubes $e,f \in X^{(1)}$, we define $e \sim f$ if their exists a sequence $e_0,\ldots,e_n \in X^{(1)}$ with $e_0 = e$ and $e_n = f$ such that for any $i>0$, either we have $e_{i-1} = e_i$ or $e_{i-1}$ and $e_i$ are the opposite edges of some $2$-cube in $X$ (i.e. there exists a $2$-cube $\sigma \,(= [0,1]^2)$ such that $e_{i-1}$ and $e_i$ are the image of the opposite edges of $\sigma$ by a gluing map to $X$). Given $e \in X^{(1)}$, we denote by $[e]$ the equivalence class of $e$ by $\sim$. An \emph{immersed hyperplane dual to $e$} is the union of all the midcubes in $X$ intersecting the $1$-cubes in $[e]$. For an immersed hyperplane $h$ in $X$, we say that $e \in X^{(1)}$ is \emph{dual to} $h$ if $h \cap e \neq \emptyset$. Given immersed hyperplanes $h$ and $k$ in $X$, we say that
    \begin{itemize}
        \item[-]
        $h$ and $k$ \emph{cross} if there exist $1$-cubes $e$ dual to $h$ and $f$ dual to $k$ such that $e$ and $f$ span a $2$-cube in $X$ (i.e. there exists a $2$-cube $\sigma \,(= [0,1]^2)$ such that $e$ and $f$ are the image of $[0,1]\times \{0\}$ and $\{0\}\times [0,1]$, respectively, by a gluing map to $X$). If $h=k$ in addition, then we say $h$ \emph{self-crosses}.
        \item[-] 
        $h$ and $k$ \emph{osculate} if there exist $1$-cubes $e$ dual to $h$ and $f$ dual to $k$ such that $e \neq f$ and $e$ and $f$ share a $0$-cube and don't span a $2$-cube in $X$. If $h=k$ in addition, then we say $h$ \emph{self-osculates}.
        \item[-]
        $h$ and $k$ \emph{interosculate} if $h$ and $k$ cross and osculate.
        \item[-]
        $h$ is \emph{two-sided} if the $1$-cubes dual to $h$ can be oriented such that any two $1$-cubes dual to $h$ and lying in a common $2$-cube in $X$ are oriented in the same direction. If $h$ is not two-sided, then we say $h$ is \emph{one-sided}.
    \end{itemize}
    A NPC cube complex $X$ is called \emph{special} if $X$ satisfies all the four conditions below.
    \begin{itemize}
        \item[(1)]
        No immersed hyperplane in $X$ self-crosses.
        \item[(2)]
        No immersed hyperplane in $X$ is one-sided.
        \item[(3)]
        No immersed hyperplane in $X$ self-osculates.
        \item[(4)] 
        No two immersed hyperplanes interosculate.
    \end{itemize}
    For NPC cube complexes $X$ and $Y$, a combinatorial map $\psi \colon X \to Y$ is called a \emph{local isometry} if for any $v \in X^{(0)}$, the induced map $\psi_v \colon \link_X(v) \to \link_Y(\psi(v))$ is injective and for any vertices $x,y \in \link_X(v)$, $x$ and $y$ are adjacent in $\link_X(v)$ if and only if $\psi_v(x)$ and $\psi_v(y)$ are adjacent in $\link_Y(\psi(v))$.
\end{defn}

\begin{defn}\label{def:CAT(0) cube complex}
    A \emph{CAT(0) cube complex} is a simply connected NPC cube complex. An immersed hyperplane of a CAT(0) cube complex is called a \emph{hyperplane}.
Let $X$ be a CAT(0) cube complex. We denote by $d_X$ the metric defined by the $1$-skeleton of $X$ (also known as $\ell^1$-metric) and denote by $\H(X)$ the set of all hyperplanes of $X$ (i.e. $\H(X) \cong X^{(1)}/\sim$). A subset $C \subset X^{(0)}$ is called \emph{convex} if for any $x,y \in C$ and any geodesic $\gamma$ from $x$ to $y$ in the $1$-skeleton of $X$, we have $\gamma^{(0)} \subset C$. 
Given $h \in \H(X)$, $X\setminus h$ has exactly two connected components. These connected components, denoted by $h^-$ and $h^+$, are called \emph{halfspaces delimited by} $h$.     Given a $0$-cube $v \in X^{(0)}$, for every $h\in \H(X)$, there exists exactly one halfspace $\h_v$ such that $v \in \h_v$. This defines an element $\alpha_v \in \prod_{h \in \H(X)}\{h^-,h^+\}$ by $\alpha_v(h)=\h_v$ for any $h \in \H(X)$. The \emph{Roller compactification $\R X$ of} $X$ is defined by
    \begin{align*}
        \R X 
        = \overline{\{\alpha_v \mid v \in X^{(0)}\}} 
        \,\Big(= \big\{\alpha\in \prod_{h \in \H(X)}\{h^-,h^+\} \,\big|\, \forall\,h,k \in \H(X),\, \alpha
    (h)\cap \alpha(k) \neq \emptyset\big\}\Big),
    \end{align*}
    where each $\{h^-,h^+\}$ admits the discrete topology of two elements and $\prod_{h \in \H(X)}\{h^-,h^+\}$ admits the product topology. Since the map $X^{(0)}\ni v \mapsto \alpha_v \in \R X$ is injective, we consider $X^{(0)} \subset \R X$ by identifying $v$ and $\alpha_v$. The \emph{Roller boundary $\partial_\R X$ of} $X$ is defined by
    \begin{align*}
        \partial_\R X = \R X \setminus X^{(0)}.
    \end{align*}
    Define a graph $\G\subset (\R X)^2$ (as an anti-reflexive, symmetric subset) by
    \begin{align}\label{eq:Borel graph on RX}
        \G = \big\{(\alpha,\beta) \in (\R X)^2 \,\big|\, \#\{h \in \H(X)\mid \alpha(h) \neq \beta(h)\}=1 \big\}.
    \end{align}
    Given subsets $A,B \subset \R X \subset \prod_{h \in \H(X)}\{h^-,h^+\}$, we say that $h \in \H(X)$ \emph{separates $A$ and $B$} if for any $\xi\in A$ and $\eta \in B$, we have $\xi(h) \neq \eta(h)$.
\end{defn}

\begin{rem}\label{rem:restriction of Borel graph}
    The restriction of $\G$ to $X^{(0)}$ coincides with the $1$-skeleton of $X$.
\end{rem}

\begin{rem}\label{rem:distance of convex sets}
    Let $X$ be a CAT(0) cube complex. For any $x \in X^{(0)}$ and any convex subset $A \subset X^{(0)}$, there exists a unique $0$-cube $y \in A$ such that $d_X(x,y) = d_X(x,A)$. For any convex sets $A,B \subset X^{(0)}$, $d_X(A,B)$ is equal to the number of hyperplanes separating $A$ and $B$. Let $a,b \in X^{(0)}$ and let $p$ be a geodesic in $X$ from $a$ to $b$, then for any distinct $1$-cubes $e_0,e_1$ of $p$, the hyperplanes $\h_0$ and $\h_1$ dual to $e_0$ and $e_1$, respectively, are distinct.
\end{rem}

\begin{defn}\label{def:RAAG}
    For a finite simple graph $\Gamma$ (i.e. no loops or multiple edges), the \emph{right angled Artin group (RAAG) $A(\Gamma)$ of} $\Gamma$ is defined by the group presentation below
    \begin{align*}
        A(\Gamma) = \la V(\Gamma) \mid vw=wv,\,(v,w) \in E(\Gamma) \ra,
    \end{align*}
    where $V(\Gamma)$ is the vertex set and $E(\Gamma)$ is the edge set of $\Gamma$. We denote by $X(\Gamma)$ the Cayley graph of $A(\Gamma)$ with respect to $V(\Gamma)$. The \emph{Salvetti complex} $R(\Gamma)$ of $\Gamma$ is the cube complex defined as follows; start with a single vertex and glue a loop to it for each vertex of $\Gamma$. Finally, for each maximal complete subgraph $\Lambda$ of $\Gamma$ with $n=\#\Lambda \ge 2$, glue a $n$-torus along the loops associated to the vertices $V(\Lambda)$.
\end{defn}

\begin{rem}\label{rem:RAAG}
    The graph $X(\Gamma)$ becomes a CAT(0) cube complex by gluing cubes. The action $A(\Gamma) \act X(\Gamma)$ is free and the quotient $X(\Gamma)/A(\Gamma)$ is $R(X)$ (see Remark \ref{rem:Deck transformation}). For any hyperplane in $X(\Gamma)$, there exists a unique vertex $v \in V(\Gamma)$ such that all the edges in $X(\Gamma)$ dual to $h$ have the label $v$ or $v^{-1}$. We call $v$ the \emph{label of} $h$.
\end{rem}

Theorem \ref{thm:special iff raag} below is from \cite[Theorem 1.1]{HW08}.

\begin{thm}[{\cite[Theorem 4.4]{Wis12}}]\label{thm:special iff raag}
    Let $X$ be a NPC cube complex with finitely many immersed hyperplanes. Then, $X$ is special if and only if there exist a finite simple graph $\Gamma$ and a local isometry $\psi \colon X\to R(\Gamma)$.
\end{thm}

In the remarks below, we recall some facts about covering maps.

\begin{rem}\label{rem:Deck transformation}
    Let $X$ be a NPC cube complex, then the universal cover $\widetilde X$ of $X$ is a CAT(0) cube complex. Let $q_X \colon \widetilde X \to X$ be the covering map and let $\widetilde x \in \widetilde X^{(0)}$ and $x \in X^{(0)}$ satisfy $q_X(\widetilde x)= x$. We denote by $\deck(X)$ the set of all homeomorphisms $f \colon \widetilde X \to \widetilde X$ such that $q_X \circ f= q_X$. The fundamental group $\pi_1(X,x)$ of $X$ based at $x$ acts on $\widetilde X$ as follows; for $z \in \widetilde X$ and $g \in \pi_1(X,x)$, take a path $\beta$ in $\widetilde X$ from $z$ to $\widetilde x$ and a loop $\alpha$ in $X$ from $x$ to $x$ with $g=[\alpha]$. Take the lift $\gamma \subset \widetilde X$ from $z$ of the path $q_X(\beta)*\alpha*q_X(\beta)^{-1}$, where $*$ denote concatenation of paths. Then, $gz \in \widetilde X$ is the other endpoint of $\gamma$ different from $z$. This action $\pi_1(X,x) \act \widetilde X$ defines the group isomorphism $\pi_1(X,x) \to \deck(X)$. Conversely, let a group $G$ act on a CAT(0) cube complex $Y$ freely, then the quotient cube complex $Y/G$ is well-defined and NPC, and the quotient map $q \colon Y \to Y/G$ is a covering map. Hence, $Y$ is the universal cover of $Y/G$ and the action $G \act Y$ defines the group isomorphism $G \to \deck(Y/G)$.
\end{rem}

\begin{lem}[{\cite[Lemma 3.12]{Wis12}}]\label{lem:loc iso implies convex emb}
    Let $X,Y$ be NPC cube complexes and $\widetilde X, \widetilde Y$ be the universal cover of $X,Y$ respectively. Let $\psi \colon X \to Y$ be a local isometry and $\widetilde\psi \colon \widetilde X \to \widetilde Y$ be a lift of $\psi$. Then, the map $\widetilde \psi$ is an embedding as a convex subcomplex.
\end{lem}

\begin{rem}\label{rem:maps between covering spaces}
    Let $X, Y, \psi, \widetilde X, \widetilde Y, \widetilde\psi$ be as in Lemma \ref{lem:loc iso implies convex emb} and let $q_X \colon \widetilde X \to X$ and $q_Y \colon \widetilde Y \to Y$ be the covering maps. Let $\widetilde x \in \widetilde X^{(0)}$, $\widetilde y \in \widetilde Y^{(0)}$, $x \in X^{(0)}$, and $y \in Y^{(0)}$ satisfy $\widetilde\psi(\widetilde x) = \widetilde y$, $\psi(x)=y$, $q_X(\widetilde x) = x$, and $q_Y(\widetilde y) = y$. Note $\psi \circ q_X = q_Y\circ\widetilde\psi$ since $\widetilde \psi$ is a lift of $\psi$. Let $\psi_* \colon \pi_1(X,x) \to \pi_1(Y,y)$ be the group homomorphism induced by $\psi$ (see Remark \ref{rem:Deck transformation}). The map $\psi_*$ is injective since $\psi$ is a local isometry. For any $g \in \pi_1(X,x)$ and $z \in \widetilde X$ we have
    \begin{align*}
        (\widetilde\psi \circ g)(z) = (\psi_*(g) \circ\widetilde\psi)(z).
    \end{align*}
\end{rem}

\begin{defn}\label{def:virtually special group}
    A cubical action of a group $G$ on a CAT(0) cube complex $X$ is called \emph{virtually cocompact special} if there exists a finite index subgroup $H$ of $G$ such that the action $H \act X$ is free and the quotient $X/H$ is a compact special cube complex. A group that admits a virtually cocompact special action on a CAT(0) cube complex is called \emph{virtually compact special}.
\end{defn}

\subsection{Hyperbolic spaces}

In this section, we review hyperbolic spaces. See \cite{BH} for more on hyperbolic spaces.

\begin{defn}
    Let $(S,d_S)$ be a metric space. For $x,y,z\in S$, we define $(x,y)_z^S$ by
\begin{align}\label{eq:gromov product}
    (x,y)_z^S=\frac{1}{2}\big( d_S(x,z)+d_S(y,z)-d_S(x,y) \big).    
\end{align}
\end{defn}

\begin{prop} \label{prop:hyp sp}
    For any geodesic metric space $(S,d_S)$, the following conditions are equivalent.
    \item[(1)]
    There exists $\delta\ge0$ satisfying the following property. Let $x,y,z\in S$, and let $p$ be a geodesic path from $z$ to $x$ and $q$ be a geodesic path from $z$ to $y$. If two points $a\in p$ and $b\in q$ satisfy $d_S(z,a)=d_S(z,b)\le (x,y)_z^S$, then we have $d_S(a,b) \le \delta$.
    \item[(2)] 
    There exists $\delta\ge0$ such that for any $w,x,y,z \in S$, we have
    \[
    (x,z)_w^S \ge \min\big\{(x,y)_w^S,(y,z)_w^S\big\} - \delta.
    \]
\end{prop}

\begin{defn}
    A geodesic metric space $S$ is called \emph{hyperbolic}, if $S$ satisfies the equivalent conditions (1) and (2) in Proposition \ref{prop:hyp sp}. We call a hyperbolic space $\delta$-\emph{hyperbolic} with $\delta \ge0$, if $\delta$ satisfies both of (1) and (2) in Proposition \ref{prop:hyp sp}. A connected graph $\Gamma$ is called \emph{hyperbolic}, if the geodesic metric space $(\Gamma,d_\Gamma)$ is hyperbolic, where $d_\Gamma$ is the graph metric of $\Gamma$.
\end{defn}

In the remainder of this section, suppose that $(S,d_S)$ is a hyperbolic geodesic metric space.

\begin{defn}\label{def:seq to infty}
    A sequence $(x_n)_{n=1}^\infty$ of elements of $S$ is said to \emph{converge to infinity}, if we have $\lim_{i,j\to\infty} (x_i,x_j)_o^S =\infty$ for some (equivalently any) $o\in S$. For two sequences $(x_n)_{n=1}^\infty,(y_n)_{n=1}^\infty$ in $S$ converging to infinity, we define the relation $\sim$ by $(x_n)_{n=1}^\infty \sim (y_n)_{n=1}^\infty$ if we have $\lim_{i,j\to\infty} (x_i,y_j)_o^S =\infty$ for some (equivalently any) $o\in S$.
\end{defn}

\begin{rem}
    It's not difficult to see that the relation $\sim$ in Definition \ref{def:seq to infty} is an equivalence relation by using the condition (2) of Proposition \ref{prop:hyp sp}.
\end{rem}

\begin{defn}
    The quotient set $\partial_\infty S$ is defined by
    \[
    \partial_\infty S = \{ {\rm sequences~in~ }S {\rm ~converging~to~infinity} \} / \sim
    \]
    and called the \emph{Gromov boundary} of $S$.
\end{defn}

\begin{rem}
    The set $\partial_\infty S$ is sometimes called the sequential boundary of $S$. Note that $\partial_\infty S$ sometimes coincides with the geodesic boundary of $S$ (e.g. when $S$ is a proper metric space), but this is not the case in general.
\end{rem}

\begin{defn}
    For $o\in S$ and $\xi,\eta \in S\cup \partial_\infty S$, we define $(\xi,\eta)_o^S$ by
    \begin{equation}\label{eq:gromov prod for boundary}
        (\xi,\eta)_o^S=\sup\big\{ \liminf_{i,j\to\infty}(x_i,y_j)_o^S ~\big|~ \xi=[(x_n)_{n=1}^\infty], \eta=[(y_n)_{n=1}^\infty] \big\},
    \end{equation}
    where we define $\xi=[(x_n)_{n=1}^\infty]$ as follows. If $\xi \in \partial_\infty S$, then $(x_n)_{n=1}^\infty$ is a sequence in $S$ converging to infinity such that $\xi$ represents the equivalence class of $(x_n)_{n=1}^\infty$. If $\xi \in S$, then $(x_n)_{n=1}^\infty$ is constant with $x_n \equiv \xi$. We define $\eta=[(y_n)_{n=1}^\infty]$ in the same way.
\end{defn}

\begin{prop} \label{prop:gromov topo}
    For any hyperbolic geodesic metric space $(S,d_S)$, there exists a unique topology $\mathcal{O}_S$ on $S\cup\partial_\infty S$ such that the relative topology of $\mathcal{O}_S$ on $S$ coincides with the metric topology of $d_S$ and for any $\xi \in \partial_\infty S$ and $o\in S$, the sets $\{U(o,\xi,n)\}_{n=1}^\infty$ defined by
    \[
    U(o,\xi,n)=\{ \eta \in S\cup\partial_\infty S \mid (\eta,\xi)_o^S > n \}
    \]
    form a neighborhood basis of $\mathcal{O}_S$ at $\xi$.
\end{prop}

\begin{rem}
    When a group $G$ acts on $S$ isometrically, this action naturally extends to the homeomorphic action on $S\cup\partial_\infty S$.
\end{rem}

\section{Hyperfiniteness in the case of right angled Artin groups}
\label{sec:Hyperfiniteness in the case of right angled Artin groups}

The goal of this section is to prove Proposition \ref{prop:hyperfinite raag}, which is about hyperfiniteness of the boundary action of right angled Artin groups. This will be used to show Theorem \ref{thm:main} in Section \ref{sec:Proof of main theorem}. We start by stating a well known fact from descriptive set theory in Lemma \ref{lem:weak Borel homo} below. We write down its proof for the convenience of the readers, although it is stated in \cite[Theorem 7.23 (i)]{Kec25} without details of the proof.

\begin{lem}\label{lem:weak Borel homo}
    Let $E$ and $F$ be Borel equivalence relations on standard Borel spaces $X$ and $Y$ respectively. Suppose that $F$ is hyperfinite and there exists a countable-to-1 Borel map $f \colon X \to Y$ such that for any $x,y \in X$, $x\, E\, y \Rightarrow f(x) \, F \, f(y)$. Then, $E$ is hyperfinite.
\end{lem}

\begin{proof}
    Define $E' \subset X^2$ by $E' = (f \times f)^{-1}(F)$. $E'$ is a Borel equivalence relation. Since $F$ is hyperfinite and $f$ is countable-to-1, both $E'$ and $F$ are countable. Note $E$ is also countable by $E \subset E'$. Since $E'$ is Borel reducible to $F$, the CBER $E'$ is hyperfinite by \cite[Proposition 5.2 (2)]{DJK94}. Hence, $E$ hyperfinite by \cite[Proposition 5.2 (1)]{DJK94}.
\end{proof}

Next, we prepare easy results on CAT(0) cube complexes from Lemma \ref{lem:grid between convex sets} up to Corollary \ref{cor:proj for conv sets RAAG}, which are used to prove Lemma \ref{lem:A(Gamma) act H_N}. These results should be well known to experts, but have not been written down in a way that we want to use to show Lemma \ref{lem:A(Gamma) act H_N} as far as I know.

In Lemma \ref{lem:grid between convex sets} and Corollary \ref{cor:proj for conv sets} below, we will always consider a combinatorial path in the $1$-skeleton of $X$ when we say a path or geodesic in $X$.

In Lemma \ref{lem:grid between convex sets}, we consider $[0,n]\times[0,m] \subset \RR^2$ as a $2$-dimensional cube complex defined naturally by integer points.

\begin{figure}[htbp]
  \begin{center}
 \hspace{0mm} 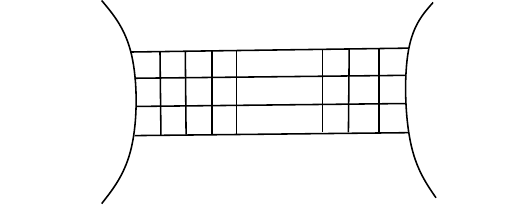
  \end{center}
   \vspace{-3mm}
  \caption{Lemma \ref{lem:grid between convex sets}}
  \label{convex}
\end{figure}

\begin{lem}\label{lem:grid between convex sets}
    Let $X$ be a CAT(0) cube complex and $A,B \subset X^{(0)}$ be convex sets with $A \cap B = \emptyset$. Let $a_0,a_1 \in A$ and $b_0,b_1 \in B$ satisfy $d_X(a_0,b_0)=d_X(a_1,b_1)=d_X(A,B)$. Then, we have $d_X(a_0,a_1)=d_X(b_0,b_1)$. Moreover, for any geodesic $p$ from $a_0$ to $b_0$ and any geodesic $q$ from $a_0$ to $a_1$ in $X$, there exist a geodesic $r$ from $a_1$ to $b_1$ in $X$, a geodesic $s$ from $b_0$ to $b_1$ in $X$, and a combinatorial map $\psi\colon [0,n]\times[0,m] \to X$, where $n = d_X(A,B)$ and $m = d_X(a_0,a_1)$, such that $\psi([0,n]\times\{0\})=p$, $\psi(\{n\}\times[0,m])=s$, $\psi([0,n]\times\{m\})=r$, and $\psi(\{0\}\times[0,m])=q$ (i.e. $\psi$ is a disk diagram for the loop $psr^{-1}q^{-1}$).
\end{lem}

\begin{proof}
    By $d_X(a_0,b_0)=d_X(a_1,b_1)=d_X(A,B)$, the median property of $X$, and convexity of $A$ and $B$, we can see for any $i \in \{0,1\}$,
    \begin{align}\label{eq:length of diagonal}
        d_X(a_i,b_{1-i}) 
        = d_X(a_i,b_i) + d_X(b_i,b_{1-i})
        = d_X(a_i,a_{1-i}) + d_X(a_{1-i}, b_{1-i}).
    \end{align}
    Take a geodesic $r'$ from $a_1$ to $b_1$ and a geodesic $s'$ from $b_0$ to $b_1$ in $X$. Let $(p_1,\ldots,p_n)$, $(q_1,\ldots,q_m)$, $(r'_1,\ldots,r'_n)$, and $(s'_1,\ldots,s'_\ell)$ be the sequences of $1$-cubes composing the geodesics $p$, $q$, $r'$, and $s'$ respectively, where $n = d_X(A,B)$, $m = d_X(a_0,a_1)$, and $\ell = d_X(b_0,b_1)$. For $p_i$ (resp. $q_i$, $r'_i$, and $s'_i$), let $\widehat{p}_i$ (resp. $\widehat{q}_i$, $\widehat{r}'_i$, and $\widehat{s}'_i$) be the hyperplane dual to $p_i$ (resp. $q_i$, $r'_i$, and $s'_i$). We have (see Remark \ref{rem:distance of convex sets})
    \begin{align}\label{eq:hyperplanes of p and r'}
        \{\widehat{p}_1,\ldots,\widehat{p}_n\} = \{\widehat{r}'_1,\ldots,\widehat{r}'_n\}.
    \end{align}
    
    Take a disk diagram $D \to X$ for the loop $ps'r^{\prime-1}q^{-1}$ with minimal area i.e. the number of $2$-cubes in $D$ is minimum among all disk diagrams of the same boundary path in $X$ (see \cite[Section 3.1]{Wis12}). By \eqref{eq:length of diagonal} and Remark \ref{rem:distance of convex sets}, every dual curve in $D$ starting from $p$ must end on $r'$. By this and \eqref{eq:hyperplanes of p and r'}, there exists a bijection $\sigma \colon \{1,\ldots,n\}\to\{1,\ldots,n\}$ such that the dual curve starting from $p_i$ ends on $r'_{\sigma(i)}$ for any $i$. 
    
    If $1<\sigma(1)$, then by $1<\sigma^{-1}(\sigma(1)-1)$, the dual curves $\alpha$ in $D$ from $p_1$ to $r'_{\sigma(1)}$ and the dual curve $\beta$ in $D$ from $p_{\sigma^{-1}(\sigma(1)-1)}$ to $r'_{\sigma(1)-1}$ cross each other. By repeating triangle moves (see the proof of \cite[Lemma 4.2]{Sag95} or \cite[Theorem 3.2]{Wis12}), we may assume that $r'$, $\alpha$, and $\beta$ form a corner (see \cite[Section 4.1]{Sag95}). Hence, the $1$-cubes $r'_{\sigma(1)-1}$ and $r'_{\sigma(1)}$ span a $2$-cube in $X$. Hence, by corner move, there exists $1$-cubes $e_0$ and $e_1$ in $X$ such that $(r'_1,\ldots,r'_{\sigma(1)-2},e_0,e_1,r'_{\sigma(1)+1},\ldots,r'_n)$ is a geodesic in $X$ from $a_1$ to $b_1$ and we have $\widehat{e}_0=\widehat{p}_{\sigma(1)}$ and $\widehat{e}_1=\widehat{p}_{\sigma(1)-1}$. By repeating this procedure, we get a geodesic $r$ in $X$ from $a_1$ to $b_1$ such that
    \begin{align}\label{eq:hp of p = hp of r}
        \forall\,i,\, \widehat{p}_i = \widehat{r}_i.
    \end{align}

    Take a disk diagram $D_1$ for the loop $ps'r^{-1}q^{-1}$ with minimal area. By \eqref{eq:hp of p = hp of r}, for any $i$, the dual curve $\alpha_i$ in $D_1$ starting from $p_i$ must end on $r_i$. Note that no distinct dual curves $\alpha_i$ and $\alpha_j$ cross since they could cross at most once by \cite[Theorem 4.4]{Sag95}.

    Since both $q$ and $s'$ are geodesic, every dual curve starting from $q$ must end on $s'$ and vice versa. Hence, we have $m=\ell$ and there exists a bijection $\tau \colon \{1,\ldots,m\}\to\{1,\ldots,m\}$ such that the dual curve starting from $q_j$ ends on $s'_{\sigma(j)}$ for any $j$. By applying corner moves to $s'$ in the same way as $p$ and $r'$, we can get a geodesic $s$ in $X$ from $b_0$ to $b_1$ such that
    \begin{align}\label{eq:hp of q = hp of s}
        \forall\,j,\, \widehat{q}_j = \widehat{s}_j.
    \end{align}
    
     Take a disk diagram $D_2$ for the loop $psr^{-1}q^{-1}$ with minimal area. By \eqref{eq:hp of q = hp of s}, for any $j$, the dual curve $\beta_j$ in $D_2$ starting from $q_j$ must end on $s_j$. Since the dual curves $\alpha_1,\ldots,\alpha_n$ and $\beta_1,\ldots,\beta_m$ form a grid, we have $D_2=[0,n]\times[0,m]$.
\end{proof}

Lemma \ref{lem:grid between convex sets} has two important corollaries, Corollary \ref{cor:proj for conv sets} and Corollary \ref{cor:proj for conv sets RAAG}. We introduce notation to state these corollaries in Definition \ref{def:rhoAB} below.

\begin{defn}\label{def:rhoAB}
    Let $X$ be a CAT(0) cube complex and $A,B \subset X^{(0)}$ be convex sets. Define $\rho_A^B \subset A$ and $\rho_B^A \subset B$ by
    \begin{align}
        \rho_A^B=\{a\in A\mid d_X(a,B)=d_X(A,B)\} 
        ~~~{\rm and}~~~ 
        \rho_B^A=\{b\in B\mid d_X(b,A)=d_X(A,B)\}.
    \end{align}
    Define a map $\psi_B^A \colon \rho_A^B \to \rho_B^A$ as follows; for $a \in \rho_A^B$, $\psi_B^A(a)$ is a unique $0$-cube in $B$ such that $d_X(a,\psi_B^A(a)) = d_X(A,B)$ (the uniqueness follows from Remark \ref{rem:distance of convex sets}).
\end{defn}

\begin{cor}\label{cor:proj for conv sets}
    Let $X$ be a CAT(0) cube complex. For any convex sets $A,B \subset X^{(0)}$, (1) and (2) below hold.
    \begin{itemize}
        \item[(1)]
        The sets $\rho_A^B$ and $\rho_B^A$ are convex and the map $\psi_B^A \colon \rho_A^B \to \rho_B^A$ is bijective.
        \item[(2)] 
        For any convex subset $C \subset A$, $\psi_B^A(C)$ is convex. 
    \end{itemize}
\end{cor}

\begin{proof}
    When $A\cap B \neq \emptyset$, the statements follow by $A \cap B = \rho_A^B = \rho_B^A$ and $\psi_B^A=\mathrm{id}_{A \cap B}$. In the following, we assume $A\cap B = \emptyset$.
    
    (1) Convexity of $\rho_A^B$ and $\rho_B^A$ follows from Lemma \ref{lem:grid between convex sets}. The map $\psi_B^A(C)$ is bijective since we can define $\psi_A^B \colon \rho_B^A \to \rho_A^B$ similarly.

    (2) Let $c_0,c_1 \in C$ and fix a geodesic $p$ from $\psi_B^A(c_0)$ to $c_0$. By Lemma \ref{lem:grid between convex sets}, for any geodesic $q$ in $X$ from $\psi_B^A(c_0)$ to $\psi_B^A(c_1)$, there exist geodesics $r$ and $s$ in $X$, respectively, from $\psi_B^A(c_1)$ to $c_1$ and from $c_0$ to $c_1$ that satisfy the property of Lemma \ref{lem:grid between convex sets}. Hence, $q^{(0)} \subset \psi_B^A(C)$.
\end{proof}

\begin{cor}\label{cor:proj for conv sets RAAG}
    Let $\Gamma$ be a finite simple graph. For any convex sets $A,B \subset X(\Gamma)^{(0)}$, (1) and (2) below hold (see Definition \ref{def:RAAG}).
    \begin{itemize}
        \item[(1)] 
        There exists $g \in A(\Gamma)$ such that for any $a \in \rho_A^B$, we have $g=a^{-1}\psi_B^A(a)$.
        \item[(2)]
        If, in addition, $A$ (resp. $B$) is a coset of $\la V_A \ra$ (resp. $\la V_B \ra$) for some $V_A \subset V(\Gamma)$ (resp. $V_B\subset V(\Gamma)$), then there exists $V_{A,B} \subset V_A \cap V_B$ such that both $\rho_A^B$ and $\rho_B^A$ are a coset of $\la V_{A,B} \ra$. Moreover, for any $U \subset V_{A,B}$ and any coset $C$ of $\la U \ra$ satisfying $C \subset A$, the set $\psi_B^A(C)$ is a coset of $\la U \ra$.
    \end{itemize}
\end{cor}

\begin{proof}
   When $A\cap B \neq \emptyset$, the statements follow by $g=1$ and $V_{A,B} = V_A \cap V_B$. In the following, we assume $A\cap B = \emptyset$.
   
   (1) Let $a_0,a_1 \in \rho_A^B$. Define $b_0 = \psi_B^A(a_0)$ and $b_1 = \psi_B^A(a_1)$ and take geodesics $p,q,r,s$ as in Lemma \ref{lem:grid between convex sets}. Then, we have $\lab(p) = \lab(r)$ since the opposite sides of a square in $X(\Gamma)$ have the same label. Hence, for any $a_0,a_1 \in \rho_A^B$, we have $a_0^{-1}\psi_B^A(a_0) = a_1^{-1}\psi_B^A(a_1)$. Thus, define $g \in A(\Gamma)$ by $g=a_0^{-1}\psi_B^A(a_0)$.

    (2) Take $g \in A(\Gamma)$ as in Corollary \ref{cor:proj for conv sets RAAG} (1) (i.e. $g=a^{-1}\psi_B^A(a)$, $a \in \rho_A^B$) and define $\supp(g) \subset V(\Gamma)$ to be the set of all letters that appear in some (equivalently, any) geodesic word of $g$ with respect to $V(\Gamma)$. Define $V_{A,B} \subset V_A \cap V_B$ by
    \begin{align*}
        V_{A,B}= \{v \in V_A \cap V_B \mid \forall\,w \in \supp(g),\, (v,w) \in E(\Gamma)\}.
    \end{align*}
    Fix $a \in \rho_A^B$ and $b \in \rho_B^A$ satisfying $d_X(a,b) = d_X(A,B)$, then we can see $a\la V_{A,B}\ra \subset \rho_A^B$ and $b\la V_{A,B}\ra \subset \rho_B^A$ by the definition of $V_{A,B}$. The converse inclusions follow from Lemma \ref{lem:grid between convex sets}. Hence, $a\la V_{A,B}\ra = \rho_A^B$ and $b\la V_{A,B}\ra = \rho_B^A$.

    Next, fix $c \in C$, then we have $C=c\la U\ra$. By Lemma \ref{lem:grid between convex sets}, we have $\psi_B^A(C) = \psi_B^A(c) \la U \ra$.
\end{proof}

Now, we start proving Proposition \ref{prop:hyperfinite raag}.

From here up to the end of Section \ref{sec:Hyperfiniteness in the case of right angled Artin groups}, let $\Gamma$ be a finite simple graph $\Gamma$, $A(\Gamma)$ be the right angled Artin group of $\Gamma$, $X(\Gamma)$ be the Cayley graph of $A(\Gamma)$ with respect to $V(\Gamma)$, $\partial_\R X(\Gamma)$ be the Roller boundary of $X(\Gamma)$, and $\H$ be the set of all hyperplanes of $X(\Gamma)$ (see Definition \ref{def:RAAG}).

We first introduce various concepts for the proof of Proposition \ref{prop:hyperfinite raag} in Definition \ref{def:for proof of raag} below.

\begin{defn}\label{def:for proof of raag}
    Let $V(\Gamma) = \{v_1,\ldots,v_N\}$ by aligning vertices of $\Gamma$. For $i \in \{1,\ldots,N\}$, define $\H_i$ to be the set of all hyperplanes of $X(\Gamma)$ whose dual edges have the label $v_i$ or $v_i^{-1}$ (see Remark \ref{rem:RAAG}). For $x,y \in \R X(\Gamma)$ (see Definition \ref{def:CAT(0) cube complex}), define $\H_i(x,y)$ and $\H(x,y)$ by
    \begin{align*}
        \H_i(x,y) &= \{h \in \H_i \mid \text{$h$ separates $x$ and $y$}\},\\
        \H(x,y) &= \{h \in \H \mid \text{$h$ separates $x$ and $y$}\}.
    \end{align*}
    For $x\in X(\Gamma)^{(0)}$, $\xi \in \partial_\R X(\Gamma)$, and $i \in I$, define an order $\le$ on $\H_i(x,\xi)$ by
    \begin{align}\label{eq:well-order on H_i(x,xi)}
        \forall\,h,k \in \H_i(x,\xi),\, h<k \iff d_X(x,h) < d_X(x,k).
    \end{align}
    Given $n \in \NN$, define an equivalence relation $F_n$ on $(\H^\NN)^n$ as follows; for $\alpha=(\alpha_i)_{i=1}^n, \beta=(\beta_i)_{i=1}^n \in (\H^\NN)^n$, where $\alpha_i,\beta_i \in \H^\NN$,
    \begin{align}\label{eq:Fn}
     \alpha \, F_n \, \beta \iff \exists\,g \in A(\Gamma),\,\forall\,i \in\{1,\ldots,n\},\,g\alpha_i\, E_t(\H)\, \beta_i.
    \end{align}
\end{defn}

\begin{rem}\label{rem:Hi}
    We have $\H = \bigsqcup_{i=1}^N \H_i$. For every $i$, the set $\H_i$ is $A(\Gamma)$-invariant and no two distinct hyperplanes in $\H_i$ cross.
\end{rem}

\begin{rem}\label{rem:well-order on H_i(x,xi)}
    The order $\le$ in \eqref{eq:well-order on H_i(x,xi)} is a well-order. Indeed, for any $h,k\in \H_i(x,\xi)$ with $h \neq k$, we have $d_X(x,h) \neq d_X(x,k)$ since $h$ and $k$ don't cross by Remark \ref{rem:Hi}.
\end{rem}

\begin{figure}[htbp]
  \begin{center}
 \hspace{0mm} 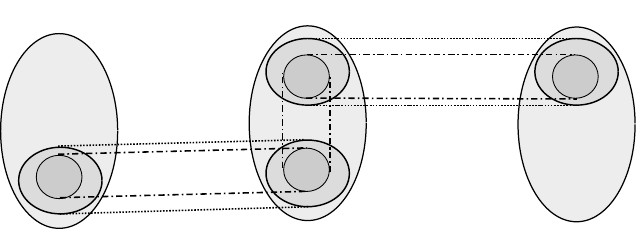
  \end{center}
   \vspace{-3mm}
  \caption{The proof of Lemma \ref{lem:A(Gamma) act H_N}}
  \label{smooth}
\end{figure}

In the proof of Lemma \ref{lem:A(Gamma) act H_N} below, we essentially use finiteness of $\Gamma$ to ensure $C_{\vec h}\neq\emptyset$ (see \eqref{eq:Cvech}).

\begin{lem}\label{lem:A(Gamma) act H_N}
    Define the action $A(\Gamma) \act \H^\NN$ by $(g,(h_n)_{n \in\NN})\mapsto (gh_n)_{n \in\NN}$, then its orbit equivalence relation $E_{A(\Gamma)}^{\H^\NN}$ is smooth.
\end{lem}

\begin{proof}
    Define a map $f \colon \H^\NN \to (A(\Gamma)^\NN)^2 \times V(\Gamma)^\NN$ as follows. Let $\vec h =(h_n)_{n \in \NN} \in \H^\NN$. For each $n \in \NN$, let $v_n \in V(\Gamma)$ be the label of the hyperplane $h_n$ (see Remark \ref{rem:RAAG}) and let $h'_n$ be the set of all directed edges with label $v_n$ (not $v_n^{-1}$) that are dual to $h_n$. Define
    \begin{align*}
        H_n=\{e_- \in X(\Gamma)^{(0)} \mid e \in h'_n\}.
    \end{align*}
    Note that $H_n$ is convex since $H_n$ is a coset of the subgroup of $A(\Gamma)$ generated by the set $\{v \in V(\Gamma)\mid (v_n,v) \in E(\Gamma)\}$. For each $n \in \NN$, define $A_n \subset H_n$, $B_n \subset H_{n+1}$, $\psi_n \colon A_n \to B_n$, and $s_n \in A(\Gamma)$ by (see Definition \ref{def:rhoAB})
    \begin{align*}
        &A_n = \rho_{H_n}^{H_{n+1}},~~~~
        B_n = \rho_{H_{n+1}}^{H_n},~~~~
        \psi_n = \psi_{H_{n+1}}^{H_n},~~~~{\rm and}\\
        &s_n = x^{-1}\psi_n(x), {\rm ~~where~~} x \in A_n.
    \end{align*}
    Recall that $s_n$ doesn't depend on $x \in A_n$ by Corollary \ref{cor:proj for conv sets RAAG} (1). Define convex sets $(C_n)_{n \in \NN}$ and $(D_n)_{n \in \NN}$ in $X(\Gamma)^{(0)}$ and $(t_n)_{n\ge2} \in A(\Gamma)^\NN$ inductively as follows. 
    
    Set $C_1=A_1$ and $D_1=B_1$. Since $D_1$ and $A_2$ are convex by Corollary \ref{cor:proj for conv sets}, we can define $C'_2 \subset D_1$, $D'_2 \subset A_2$, $\varphi_2 \colon C'_2 \to D'_2$ and $t_2\in A(\Gamma)$ by (see Definition \ref{def:rhoAB} and Corollary \ref{cor:proj for conv sets RAAG})
    \begin{align*}
        &C'_2 = \rho_{D_1}^{A_2},~~~~
        D'_2 = \rho_{A_2}^{D_1},~~~~
        \varphi_2 = \psi_{A_2}^{D_1},~~~~{\rm and}\\
        &t_2 = x^{-1}\varphi_2(x), {\rm ~~where~~} x \in C'_2.
    \end{align*}
    Define $C_2 = \psi_1^{-1} (C'_2)$ and $D_2 = \psi_2 (D'_2)$. The sets $C_2$ and $D_2$ are convex by Corollary \ref{cor:proj for conv sets} and we have $C_2 \subset C_1=A_1$, $D_2 \subset  B_2$, and $D_2 = (\psi_2\circ\varphi_2\circ\psi_1) (C_2)$. By repeating this procedure (e.g. next, apply Definition \ref{def:rhoAB} to $D_2$ and $A_3$), we get convex sets $(C_n)_{n \ge 2}$, $(D_n)_{n \ge 2}$, $(C'_n)_{n \ge 2}$, and $(D'_n)_{n \ge 2}$ in $X(\Gamma)^{(0)}$, maps $(\varphi_n)_{n\ge 2}$, and $(t_n)_{n\ge 2}$ such that for any $n \ge 2$,
    \begin{align*}
        &C'_{n+1} = \rho_{D_n}^{A_{n+1}} \subset D_n,~~~~
        D'_{n+1} = \rho_{A_{n+1}}^{D_n} \subset A_{n+1},~~~~
        \varphi_{n+1}=\psi_{A_{n+1}}^{D_n} \colon C'_{n+1} \to D'_{n+1},~~~~\\
        &C_{n+1}=(\psi_n\circ\varphi_n \circ \cdots \circ \psi_2\circ\varphi_2\circ\psi_1)^{-1}(C'_{n+1}) \subset C_n,~~~~\\
        &D_{n+1}=\psi_{n+1}(D'_{n+1}) \subset B_{n+1},~~~~\\
        &t_{n+1}=x^{-1}\varphi_{n+1}(x), {\rm ~~where~~} x \in C'_{n+1}.
    \end{align*}   
    By Corollary \ref{cor:proj for conv sets RAAG} (2), there exist subsets $\{V_n\}_{n\in\NN}$ of $V(\Gamma)$ such that for any $n \in \NN$, $C_n$ is a coset of $\la V_n \ra$ and we have $V_{n+1} \subset V_n$. Since $V(\Gamma)$ is finite, there exists $N \in \NN$ such that for any $n \ge N$, $V_n=V_N$, which implies $C_n=C_N$. Define $C_{\vec h}$ by
    \begin{align}\label{eq:Cvech}
        C_{\vec h}=\bigcap_{n\in\NN}C_n,
    \end{align}
    then $C_{\vec h}\neq\emptyset$ by $C_{\vec h} = C_N$ and the set $(\psi_n\circ\varphi_n \circ \cdots\circ \psi_2\circ\varphi_2\circ\psi_1)(C_{\vec h})$ in $B_n$ is well-defined for any $n \ge 1$.
    
    Define $f(\vec h) = ((s_n)_{n \ge 1}, (t_n)_{n \ge 2}, (v_n)_{n\in\NN})$. We claim that for any $\vec h , \vec k \in \H^\NN$, we have $\vec h \,E_{A(\Gamma)}^{\H^\NN}\, \vec k \,\Leftrightarrow\, f(\vec h) = f(\vec k)$. The direction $\vec h \,E_{A(\Gamma)}^{\H^\NN}\, \vec k \Rightarrow f(\vec h) = f(\vec k)$ is straightforward since the action $A(\Gamma) \act X(\Gamma)$ preserves edge labels.

    We'll show the converse direction. Let $\vec  h =(h_n)_{n \in \NN}, \vec  k =(k_n)_{n \in \NN} \in \H^\NN$ satisfy $f(\vec h) = f(\vec k) = ((s_n)_{n \ge 1}, (t_n)_{n \ge 2}, (v_n)_{n\in\NN})$. Fix $x \in C_{\vec h}$ and $y \in C_{\vec k}$, which is possible by $C_{\vec h}, C_{\vec k} \neq \emptyset$ (see \eqref{eq:Cvech}). Define $g \in A(\Gamma)$ by $g=yx^{-1}$ (i.e. $gx=y$). Note that $h_1$ (resp. $k_1$) is the unique hyperplane dual to the edge whose initial vertex is $x$ (resp. $y$) and whose label is $v_1$. Hence, $gh_1=k_1$. For any $n \ge 2$, define $x_n,y_n\in A(\Gamma)$ by
    \begin{align*}
        x_n = x(s_1t_2s_2 \cdots t_{n-1}s_{n-1})
        {\rm ~~~and~~~}
        y_n = y(s_1t_2s_2 \cdots t_{n-1}s_{n-1}),
    \end{align*}
    then $h_n$ (resp. $k_n$) is the unique hyperplane dual to the edge whose initial vertex is $x_n$ (resp. $y_n$) and whose label is $v_n$. Hence, we have $gh_n=k_n$ by $gx_n=y_n$. Thus, $g\vec h = \vec k$.
\end{proof}

\begin{cor}\label{cor:E_n hyperfinite}
    For every $n \in \NN$, $F_n$ is a hyperfinite CBER (see \eqref{eq:Fn}).
\end{cor}

\begin{proof}
    First, we show that $F_1$ is hyperfinite (see \eqref{eq:Fn}). By Lemma \ref{lem:A(Gamma) act H_N}, there exist a standard Borel space $X$ and a Borel map $f \colon \H^\NN \to X$ such that $\vec h \, E_{A(\Gamma)}^{\H^\NN} \, \vec k \Leftrightarrow f(\vec h) = f(\vec k)$ for any $\vec h, \vec k \in \H^\NN$. Define a map $\varphi \colon \H^\NN \to X^\NN$ as follows; for any $\vec h = (h_n)_{n \in \NN} \in \H^\NN$,
    \begin{align*}
        \varphi(\vec h) = (f((h_i)_{i \ge n}))_{n \in \NN}.
    \end{align*}
    For any $\vec h, \vec k \in \H^\NN$, we have $\vec h\,F_1\, \vec k \Leftrightarrow \varphi(\vec h) \, E_t(X)\, \varphi(\vec k)$. Hence, $F_1$ is hypersmooth by Theorem \ref{thm:DJK}. Since $F_1$ is countable as well, $F_1$ is hyperfinite by \cite[Theorem 5.1 (1)$\,\Leftrightarrow\,$(3)]{DJK94}.
    
    Next, let $n \in \NN$. Since $F_1$ is hyperfinite, $F_1\times \cdots\times F_1$ on $(\H^\NN)^n$ is hyperfinite by \cite[Proposition 5.2 (5)]{DJK94}. By $F_n \subset F_1\times \cdots\times F_1$ and \cite[Proposition 5.2 (1)]{DJK94}, $F_n$ is hyperfinite.
\end{proof}

\begin{lem}\label{lem:well-order on H_i(x,xi)}
    Let $x,y\in X(\Gamma)^{(0)}$, $\xi \in \partial_\R X(\Gamma)$, and $i \in \{1,\ldots,N\}$. Then, for any $h,k\in \H_i(x,\xi)\cap \H_i(y,\xi)$, we have $d_X(x,h) < d_X(x,k) \iff d_X(y,h) < d_X(y,k)$.
\end{lem}

\begin{proof}
    Suppose $d_X(x,h) < d_X(x,k)$, then $k$ is on the opposite halfspace of $h$ from $x$. Also, both $x$ and $y$ are on the opposite halfspace of $h$ from $\xi$. Hence, $k$ and $\xi$ are on the same halfspace of $h$ and hence $y$ is on the opposite halfspace of $h$ from $k$. Thus, $d_X(y,h) < d_X(y,k)$. The converse direction follows by swapping $x$ and $y$.
\end{proof}

\begin{prop}\label{prop:hyperfinite raag}
     The CBER $E_{A(\Gamma)}^{\partial_\R X(\Gamma)}$ induced by the action $A(\Gamma) \act \partial_\R X(\Gamma)$ is hyperfinite.
\end{prop}

\begin{proof}
    Fix $o \in X(\Gamma)^{(0)}$. Let $\I$ be the set of all non-empty subsets of $\{1,\ldots,N\}$. For $I\in \I$, define $\partial_\R X(\Gamma)_I$ by
    \begin{align*}
        \partial_\R X(\Gamma)_I = \{\xi \in \partial_\R X(\Gamma) \mid i\in I \iff |\H_i(o,\xi)|=\infty\}.
    \end{align*}
    $\partial_\R X(\Gamma)_I$ is $A(\Gamma)$-invariant since the action $A(\Gamma) \act X(\Gamma)$ preserves labels of edges. For any $i \in \{1,\ldots,N\}$ and $\xi \in \partial_\R X(\Gamma)$, we have
    \begin{align*}
        |\H_i(o,\xi)|<\infty \iff \text{$\exists\,H \subset \H_i$ : finite, $\big[h \in H \iff h \in \H(o,\xi)\big]$}.
    \end{align*}
    Hence, we can see that $\partial_\R X(\Gamma)_I$ is Borel for any $I\in \I$. Also, $\partial_\R X(\Gamma) = \bigsqcup_{I \in \I}\partial_\R X(\Gamma)_I$. Thus, it's enough to show that $E_{A(\Gamma)}^{\partial_\R X(\Gamma)}|_{\partial_\R X(\Gamma)_I}$ is hyperfinite for any $I \in \I$. 
    
    Let $I \in \I$. Define a map $f \colon \partial_\R X(\Gamma)_I \to (\H^\NN)^I$ as follows. Let $\xi \in \partial_\R X(\Gamma)_I$. For each $i \in I$, align elements of $\H_i(o,\xi)$ by the well-order as in \eqref{eq:well-order on H_i(x,xi)}, which we denote $\H_i(o,\xi) = \{h_{i,1}<h_{i,2}<\cdots\}$ (see Remark \ref{rem:well-order on H_i(x,xi)}). Define $f(\xi)=\big((h_{i,n})_{n \in \NN}\big)_{i \in I}$. It's not difficult to see that the map $f$ is Borel.

    If $\xi,\eta \in \partial_\R X(\Gamma)_I$ satisfy $f(\xi) = f(\eta)$, then the symmetric difference $\H(o,\xi)\triangle \H(o,\eta)$ is finite. Thus, $f$ is countable-to-1 since $\H$ is countable.

    Let $\xi,\eta \in \partial_\R X(\Gamma)_I$ satisfy $\xi\,E_{A(\Gamma)}^{\partial_\R X(\Gamma)}\,\eta$, then there exists $g \in A(\Gamma)$ such that $g \xi = \eta$. Let $f(\xi)=\big((h_{i,n})_{n \in \NN}\big)_{i \in I}$ and $f(\eta)=\big((k_{i,n})_{n \in \NN}\big)_{i \in I}$. Given $i \in I$, we have $\H_i(go,\eta) = \H_i(go,g\xi)=\{gh_{i,1}<gh_{i,2}<\cdots\}$ for any $i \in I$ since the action $A(\Gamma) \act X(\Gamma)$ preserves edge labels. By $|\H_i(o,\eta)\triangle \H_i(go,\eta)|< \infty$ and Lemma \ref{lem:well-order on H_i(x,xi)}, we can see $(k_{i,n})_{n \in \NN} \,E_t(\H)\, (gh_{i,n})_{n \in \NN}$. Thus, $f(\xi) \,F_{|I|}\, f(\eta)$ (see \eqref{eq:Fn}). By Lemma \ref{lem:weak Borel homo} and Corollary \ref{cor:E_n hyperfinite}, $E_{A(\Gamma)}^{\partial_\R X(\Gamma)_I}$ is hyperfinite, hence so is $E_{A(\Gamma)}^{\partial_\R X(\Gamma)}$.
\end{proof}

\begin{rem}
    Let $\G$ be the graph on $\R X(\Gamma)$ as in \eqref{eq:Borel graph on RX}. Since $\G$ is $A(\Gamma)$-invariant, the group $A(\Gamma)$ acts on the quotient $\partial_\R X(\Gamma) / \G$ (i.e. $\partial_\R X(\Gamma) / \G$ is the set of all connected components of $\G$ except $X(\Gamma)^{(0)}$). By \cite[Theorem 1.1]{Oya25}, $\partial_\R X(\Gamma) / \G$ is a standard Borel space. In the same way as the proof of Proposition \ref{prop:hyperfinite raag}, we can show that $E_{A(\Gamma)}^{\partial_\R X(\Gamma)/ \G}$ is hyperfinite, which implies hyperfiniteness of $E_{A(\Gamma)}^{\partial_\R X(\Gamma)}$ by Lemma \ref{lem:weak Borel homo} since the quotient map $\partial_\R X(\Gamma) \to \partial_\R X(\Gamma) / \G$ is $A(\Gamma)$-equivariant and countable-to-1.
\end{rem}

\section{Proof of Theorem \ref{thm:main}}\label{sec:Proof of main theorem}

The goal of this section is to prove Theorem \ref{thm:main}. In fact, we prove it in a more general form, which is Theorem \ref{thm:hyperfinite of virt special groups} below. Note that in Theorem \ref{thm:hyperfinite of virt special groups}, $X$ is countable since we can show that any finite dimensional CAT(0) cube complex with countably many hyperplanes is countable by induction on the dimension. Therefore, $\partial_\R X$ is a Polish space since it is a $G_\delta$ subset of the compact metrizable space $\R X$.

\begin{thm}\label{thm:hyperfinite of virt special groups}
    Let $X$ be a CAT(0) cube complex and $G$ be a countable group acting cubically on $X$. Suppose that there exists a finite index subgroup $H$ of $G$ such that $H \act X$ is free and the quotient $X/H$ is a special cube complex with finitely many immersed hyperplanes. Then, $E_G^{\partial_\R X}$ is hyperfinite.
\end{thm}

\begin{proof}
    Since $H$ is a finite index subgroup of $G$, we have $E_H^{\partial_\R X} \subset E_G^{\partial_\R X}$ and every $E_G^{\partial_\R X}$-class contains only finitely many $E_H^{\partial_\R X}$-classes. Hence, it's enough to show that $E_H^{\partial_\R X}$ is hyperfinite by Proposition \ref{prop:JKL}.
    
    Set $Y = X/H$. By Theorem \ref{thm:special iff raag}, there exist a finite simple graph $\Gamma$ and a local isometry $\psi \colon Y \to R(\Gamma)$ (see Definition \ref{def:RAAG}). Note that $X$ and $X(\Gamma)$ are the universal cover of $Y$ and $R(\Gamma)$ respectively (see Remark \ref{rem:RAAG} and Remark \ref{rem:Deck transformation}). Fix $y \in Y^{(0)}$ and $o \in R(\Gamma)^{(0)}$ satisfying $\psi(y) = o$ and define the actions $\pi_1(Y,y) \act X$ and $\pi_1(R(\Gamma),o) \act X(\Gamma)$ as in Remark \ref{rem:Deck transformation}. The local isometry $\psi$ induces the embedding $\psi_* \colon \pi_1(Y,y) \to \pi_1(R(\Gamma),o)$. By Remark \ref{rem:RAAG} and Remark \ref{rem:Deck transformation}, we have 
    \begin{align}\label{eq:E H R X}
        E_H^{\partial_\R X} = E_{\pi_1(Y,y)}^{\partial_\R X}
        {\rm~~~and~~~}
        E_{A(\Gamma)}^{\partial_\R X(\Gamma)} = E_{\pi_1(R(\Gamma),o)}^{\partial_\R X(\Gamma)}.
    \end{align}
    There exists a lift $\widetilde \psi \colon X \to X(\Gamma)$ of $\psi$ since $X$ is simply connected. By Lemma \ref{lem:loc iso implies convex emb}, the map $\widetilde \psi$ is an embedding and $\widetilde\psi(X)$ is a convex subcomplex. Hence, $\widetilde \psi$ induces the injection $\widehat \psi \colon \partial_\R X \to \partial_\R X(\Gamma)$.
    By Remark \ref{rem:maps between covering spaces}, $\widehat \psi (\partial_\R X)$ is invariant under the action of $\psi_*(\pi_1(Y,y))$ and we have
    \begin{align}\label{eq:psi induces Borel red}
        (\widehat \psi \times \widehat \psi)(E_{\pi_1(Y,y)}^{\partial_\R X}) = E_{\psi_*(\pi_1(Y,y))}^{\widehat \psi (\partial_\R X)}.
    \end{align}
    By \eqref{eq:E H R X} and Proposition \ref{prop:hyperfinite raag}, $E_{\pi_1(R(\Gamma),o)}^{\partial_\R X(\Gamma)}$ is hyperfinite. Hence, $E_{\psi_*(\pi_1(Y,y))}^{\widehat \psi (\partial_\R X)}$ is hyperfinite by \cite[Proposition 5.2 (1) and (3)]{DJK94}. Hence, by \eqref{eq:E H R X} and \eqref{eq:psi induces Borel red}, $E_H^{\partial_\R X}$ is hyperfinite.
\end{proof}

\begin{proof}[Proof of Theorem \ref{thm:main}]
    This follows from Theorem \ref{thm:hyperfinite of virt special groups} since every compact NPC cube complex has only finitely many immersed hyperplanes.
\end{proof}

\begin{rem}
    Note that in Theorem \ref{thm:hyperfinite of virt special groups}, the group $G$ doesn't have to be finitely generated. We can actually construct a special cube complex with finitely many hyperplanes whose fundamental group is not finitely generated by modifying the Wise's construction in \cite[Remark 3.11]{CS11} as follows. Let $A$ be the graph obtained by taking two copies $R_1$ and $R_2$ of the real line $\RR$ and gluing together along the integers ($A$ becomes a bi-infinite chain of bigons). We attach two bi-infinite rectangular strips $S_1$ and $S_2$ to $A$ as follows, where each strip $S_i$ ($i=1,2$) is identified as $\RR \times [0,1]$ for the labeling purpose. The first strip $S_1$ is attached by gluing the top side $[n,n+1]\times\{1\}$ of $S_1$ to the line segment $[n,n+1]$ of $R_1$ in $A$ and gluing the bottom side $[n,n+1]\times \{0\}$ of $S_1$ to the line segment $[n+1,n+2]$ of $R_2$ in $A$ for each $n \in \ZZ$. The second strip $S_2$ is attached by gluing the top side $[n,n+1] \times \{1\}$ of $S_2$ to the line segment $[n+1,n+2]$ of $R_1$ in $A$ and gluing the bottom side $[n,n+1] \times \{0\}$ of $S_2$ to the line segment $[n,n+1]$ of $R_2$ in $A$ for each $n \in \ZZ$. In this way, we obtain a special cube complex $X$ with exactly 4 hyperplanes and whose fundamental group is the free group $F_\infty$ with countably many generators (note that $X$ is homeomorphic to the cube complex in \cite[Remark 3.11]{CS11}).
\end{rem}

We finish this section with 4 immediate yet interesting corollaries of Theorem \ref{thm:main}.

\begin{cor}\label{cor:cubulated hyp. gp}
    If a hyperbolic group $G$ acts on a CAT(0) cube complex $X$ properly and cocompactly, then $E_G^{\partial_\R X}$ is hyperfinite.
\end{cor}

\begin{proof}
    This follows from Theorem \ref{thm:main} and \cite[Theorem 1.1]{Ian13}
\end{proof}

In Corollary \ref{cor:contact graph} below, note that the contact graph of every CAT(0) cube complex is a quasi-tree, hence hyperbolic by \cite{Hag14}.

\begin{cor}[Theorem {\ref{thm:intro contact graph}}]\label{cor:contact graph}
    Let $X$ be a CAT(0) cube complex and $G$ be a countable group acting virtually cocompact specially on $X$. Let $\mathscr{C} X$ be the contact graph of $X$. Then, the orbit equivalence relation $E_G^{\partial_\infty \mathscr{C} X}$ induced by the action of $G$ on the Gromov boundary $\partial_\infty \mathscr{C} X$ of $\mathscr{C} X$ is hyperfinite.
\end{cor}

\begin{proof}
    Note that $X$ is uniformly locally finite, hence finite dimensional. By \cite[Theorem 1.2]{FLM24}, there exists an $\aut(X)$-equivariant homeomorphism $f$ from the regular boundary $\partial_{reg}X$ of $X$ to $\partial_\infty \mathscr{C} X$. Since $\partial_{reg}X$ is an $\aut(X)$-invariant Borel subset of $\partial_\R X$, $E_G^{\partial_{reg} X}$ is hyperfinite by Theorem \ref{thm:main} and \cite[Proposition 5.2 (3)]{DJK94}. Thus, $E_G^{\partial_\infty \mathscr{C} X}$ is hyperfinite by $f$.
\end{proof}

\begin{cor}[Theorem {\ref{thm:intro RAAG}}]
    Let $\Gamma$ be a finite simple graph, $A(\Gamma)$ be the right angled Artin group of $\Gamma$, and $\Gamma^e$ be the extension graph of $\Gamma$. Then, the orbit equivalence relation $E_{A(\Gamma)}^{\partial_\infty\Gamma^e}$ induced by the action of $A(\Gamma)$ on the Gromov boundary $\partial_\infty \Gamma^e$ of $\Gamma^e$ is hyperfinite.
\end{cor}

\begin{proof}
    This follows from Corollary \ref{cor:contact graph} and the fact that there exists a $A(\Gamma)$-equivariant quasi-isometry from $\Gamma^e$ to the contact graph $\mathscr{C} X(\Gamma)$ (see Definition \ref{def:RAAG} and \cite[Section 7]{KK14}).
\end{proof}

\begin{cor}[Theorem {\ref{thm:intro Coxeter groups}}]
    Let $G$ be a finitely generated Coxeter group and let $C$ be the Niblo-Reeves CAT(0) cube complex on which $G$ acts properly. Then, the orbit equivalence relation $E_G^{\partial_\R C}$ induced by the action of $G$ on the Roller boundary $\partial_\R C$ of $C$ is hyperfinite.
\end{cor}

\begin{proof}
    By \cite[Theorem 1.2]{HW10}, there exists a finite index torsion-free subgroup $H$ of $G$ such that the quotient $C/H$ is special. Since the action $G \act C$ has finitely many hyperplane orbits (see \cite[Remark 6.3]{HW10}) and $H$ is finite index in $G$, the action $H \act C$ also has finitely many hyperplane orbits. Thus, the statement follows by Theorem \ref{thm:hyperfinite of virt special groups}.
\end{proof}

\section{Cubulated hyperbolic groups}\label{sec:Cubulated hyperbolic group}

In this section, we verify that Theorem \ref{thm:main} can be considered as a generalization of \cite{HSS20}. The goal of this section is to prove Proposition \ref{prop:hyperfinite of Roller and Gromov}. We start by proving a descriptive set theoretic result that we use in the proof of Proposition \ref{prop:hyperfinite of Roller and Gromov}.

\begin{lem}\label{lem:hyperfininte when finite to 1}
    Let $X$ and $Y$ be standard Borel spaces and $G$ be a countable group acting on $X$ and $Y$ as Borel isomorphisms. Suppose that there exists a $G$-equivariant, finite-to-1, surjective Borel map $f \colon X \to Y$. Then, $E_G^X$ is hyperfinite if and only if $E_G^Y$ is hyperfinite.
\end{lem}

\begin{proof}
    First, suppose that $E_G^Y$ is hyperfinite, then $E_G^X$ is hyperfinite by Lemma \ref{lem:weak Borel homo}.
    
    Next, suppose that $E_G^X$ is hyperfinite. Since $f$ is finite-to-1 and surjective, by applying Theorem \ref{thm:Lusin-Novikov} and Theorem \ref{thm:Arsenin-Kunugui} to the set $\{(x,f(x)) \in X\times Y \mid x \in X\}$, there exists a Borel subset $A \subset X$ such that $f|_A \colon A\to Y$ is Borel isomorphic. Since $E_G^X$ is hyperfinite, $E_G^X|_A$ is hyperfinite. Define $F \subset Y^2$ by $F = (f \times f)(E_G^X|_A)$, then $F$ is a hyperfinite Borel equivalence relation on $Y$. By $G$-equivariance of $f$, we have $F \subset E_G^Y$.

    We claim that every $E_G^Y$-class contains only finitely many $F$-classes. Let $y_0 \in Y$ and set $n=|f^{-1}(y_0)|$. Since $f$ is surjective and finite-to-1, we have $0 < n <\infty$. Suppose for contradiction that there exist $y_0,\ldots,y_n \in Y$ such that for any distinct $i,j \in \{0,\ldots,n\}$, we have $y_i\, E_G^Y \, y_j$ and $\neg\, y_i\, F \, y_j$. For each $i \in \{0,\ldots,n\}$, take a unique $x_i \in A$ such that $f(x_i)=y_i$. Also, by $y_i\, E_G^Y \, y_0$, there exists $g_i \in G$ such that $g_iy_i=y_0$. Note $f(g_ix_i)= g_if(x_i)=y_0$ i.e. $g_ix_i\in f^{-1}(y_0)$. Hence, by $n=|f^{-1}(y_0)|$, there exist distinct $i,j \in \{0,\ldots,n\}$ such that $g_ix_i=g_jx_j$, hence $x_i \,E_G^X\, x_j$. This implies $y_i \,F\, y_j$ by $x_i,x_j \in A$, which contradicts the way we took $y_0,\ldots,y_n$.

    Thus, by the above claim and Proposition \ref{prop:JKL}, $E_G^Y$ is hyperfinite.
\end{proof}

Next, we record results on hyperbolic CAT(0) cube complexes to relate the Gromov boundary and the Roller boundary from Lemma \ref{lem:geod rays of same Roller limit} up to Lemma \ref{lem:unif. loc .fin. hyp.}. These results are easy to show and should be well known to experts, but we write down the proofs for completeness.

We introduce some notation for Lemma \ref{lem:geod rays of same Roller limit} and Lemma \ref{lem:from Roller bdry to Gromov bdry} below. Let $X$ be a CAT(0) cube complex. A path $p=(p_0,p_1,\ldots)$ in $X$ is a combinatorial path, where $p_{i-1}$ and $p_i$ are adjacent $0$-cubes for every $i\in\NN$. We denote by $d_\G$ the graph metric on $\R X$ defined by $\G$ (see \eqref{eq:Borel graph on RX}). For $x,y \in \R X$, we denote by $\H(x,y)$ the set of all hyperplanes in $X$ separating $x$ and $y$. For $x \in \R X$, we define $\lk_\G(x)=\{y \in \R X \mid (x,y)\in \G\}$ (note Remark \ref{rem:restriction of Borel graph}).

\begin{figure}[htbp]
  \begin{center}
 \hspace{0mm} 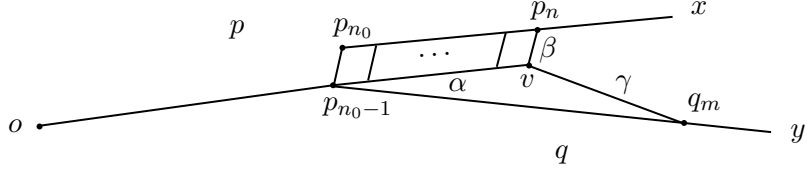
  \end{center}
   \vspace{-3mm}
  \caption{The proof of Lemma \ref{lem:geod rays of same Roller limit} (3)}
  \label{geodray}
\end{figure}

\begin{lem}\label{lem:geod rays of same Roller limit}
    Let $X$ be a CAT(0) cube complex. The following hold.
    \begin{itemize}
        \item[(1)]
        Let $p=(p_0,p_1,\ldots)$ and $q=(q_0,q_1,\ldots)$ be geodesic rays in $X$ with $o=p_0=q_0 \in X^{(0)}$. If $p$ and $q$ converge to the same point in $\partial_\R X$, then for any $n \in \NN$, there exists $m \in \NN$ such that $d_X(o,p_n) + d_X(p_n,q_m) = d_X(o,q_m)$.
        \item[(2)] 
        For any $x,y \in \partial_\R X$, we have $d_\G(x,y)=|\H(x,y)|$.
        \item[(3)]
        For any $x,y \in \partial_\R X$ with $\H(x,y)=\{h\}$ and any geodesic ray $p=(p_0,p_1,\ldots)$ in $X$ converging to $x$, there exist $n,m \in \NN$ and a geodesic ray $q=(q_0,q_1,\ldots)$ in $X$ converging to $y$ such that for any $i \in\NN$, we have $d_X(p_{n+i},q_{m+i})=1$ and the edge $(p_{n+i},q_{m+i})$ is dual to $h$.
        \item[(4)]
        If $D=\sup_{x \in X^{(0)}}|\lk_\G(x)| < \infty$, then $\sup_{x \in \partial_\R X}|\lk_\G(x)|\le D-2$.
    \end{itemize}
\end{lem}

\begin{proof}
    (1)
    For each $i \in \NN$, let $h_i$ (resp. $k_i$) be the hyperplane dual to the edge $(p_{i-1},p_i)$ (resp. $(q_{i-1},q_i)$). We have $\{h_i \mid i \in\NN\} = \{k_i \mid i \in\NN\}$, since $p$ and $q$ converge to the same point in $\partial_\R X$. For any $n \in \NN$, there exists $m \in \NN$ such that $\{h_i \mid i \le n\} \subset \{k_i \mid i \le m\}$. By considering the median of the three vertices $o$, $p_n$, and $q_m$, we can see $d_X(o,q_m) = d_X(o,p_n) + d_X(p_n,q_m)$.

    (2)
    $|\H(x,y)| \le d_\G(x,y)$ follows by $|\H(x,y)|=1 \,\Leftrightarrow\, d_\G(x,y)=1$. We will show $d_\G(x,y)\le|\H(x,y)|$. Since the case $|\H(x,y)|=\infty$ is trivial, we assume $|\H(x,y)|<\infty$. Define an order $\le$ on $\H(x,y)$ as follows; for $h,k\in \H(x,y)$, $h\le k \,\Leftrightarrow\, x(h) \subset x(k)$. Since $\H(x,y)$ is finite, there exists a minimal element $h_1 \in \H(x,y)$. Define $z\in \prod_{h \in \H}\{h^-,h^+\}$ by $z(h_1) = y(h_1)$ and $z(h)=x(h)$ if $h\in \H\setminus\{ h_1\}$.

    We claim $z \in \partial_\R X$ i.e. $\forall\,h,k \in \H,\, z(h) \cap z(k) \neq \emptyset$. By $x \in \partial_\R X$, this is trivial if either $h,k \in \H\setminus\{h_1\}$ or $h$ and $k$ cross. If $k \in \H\setminus\H(x,y)$, then by $z(k) =x(k) =y(k)$ and $y \in \partial_\R X$, we have $z(h_1)\cap z(k) = y(h_1)\cap y(k) \neq \emptyset$. If $k \in \H(x,y)\setminus\{h_1\}$ and $k$ and $h_1$ don't cross (the case $k$ and $h_1$ cross is trivial as mentioned above), then by minimality of $h_1$, we have $x(h_1) \subsetneq x(k)$, hence $z(h_1) \cap z(k) = y(h_1) \cap x(k) \neq \emptyset$.

    Thus, $z \in\partial_\R X$ is adjacent with $x$ in $\G$ and we have $\H(z,y) = \H(x,y) \setminus\{h_1\}$. By repeating this, we get a path of length $|\H(x,y)|$ in $\G$ from $x$ to $y$. Hence, $d_\G(x,y)\le|\H(x,y)|$.

    (3) Fix $o \in X^{(0)}$. By $\H(x,y)=\{h\}$, exactly one of $\H(o,x)$ or $\H(o,y)$ contains $h$. We may assume $h \in \H(o,x)$. Indeed, if $h \in \H(o,y)$, then we can take $o' \in X^{(0)}$ with $h\in \H(o,o')$. Note $h \in \H(o',x)$. There exists $N_0\in \NN$ such that for any geodesic $\gamma$ in $X$ from $o'$ to $p_{N_0}$, the path $\gamma\cdot p_{[p_{N_0},\infty)}$ is a geodesic ray. It's enough to show the statement for $\gamma\cdot p_{[p_{N_0},\infty)}$ since the path merges $p$.
    
    Hence, we assume $h \in \H(o,x)$. There exists $n_0 \in \NN$ such that the edge $(p_{n_0-1},p_{n_0})$ is dual to $h$. By \cite[Proposition A.2]{Gen20}, we can take a geodesic ray $q=(q_0,q_1,\ldots)$ from $p_{n_0-1}$ in $X$ converging to $y$ (note $q_0=p_{n_0-1}$). For each $i \in \NN$, let $h_i$ (resp. $k_i$) be the hyperplane dual to the edge $(p_{i-1}, p_i)$ (resp. $(q_{i-1}, q_i)$). Note $h_{n_0}=h$ and $\{h_i\}_{i\ge n_0+1} = \{k_i\}_{i\ge 1}$. For any $n > n_0$, their exists $m \in \NN$ such that $\{h_i\}_{i = n_0+1}^n \subset \{k_i\}_{i = 1}^m$. Let $v \in X^{(0)}$ be the median of $p_{n_0-1}$, $p_n$, and $q_m$. Take geodesics $\alpha$, $\beta$, and $\gamma$ in $X$, respectively, from $p_{n_0-1}$ to $v$, from $v$ to $p_n$, and from $v$ to $q_m$. By $\{h_i\}_{i = n_0+1}^n \subset \{k_i\}_{i = 1}^m$ and $h \notin \{k_i\}_{i = 1}^m$, the set of all hyperplanes crossed by $\alpha$ is exactly $\{h_i\}_{i = n_0+1}^n$ and the set of all hyperplanes crossed by $\beta$ is exactly $\{h\}$. Hence, $\beta$ is an edge dual to $h$, and we may assume that the geodesics $p_{[p_{n_0},p_n]}$ and $\alpha$ are the opposite sides of a strip of $2$-cubes from the edge $p_{[p_{n_0-1},p_{n_0}]}$ to $\beta^{-1}$ by deforming $\alpha$ by corner moves (in the same way as the proof of Lemma \ref{lem:grid between convex sets}). Also, the paths $\beta \cdot p_{[p_n,\infty)}$ and $\alpha\cdot\gamma\cdot q_{[q_m,\infty)}$ are geodesic rays. Note that $\alpha\cdot\gamma\cdot q_{[q_m,\infty)}$ is obtained from $q$ by corner moves. 
    
    By taking larger and larger $n$ and repeating this argument, we get a geodesic ray $q'$ from $p_{n_0-1}$ converging to $y$ that is the limit of sequences of paths obtained from $q$ by corner moves and satisfies the statement.
    
    (4) Let $x \in \partial_\R X$. By \cite[Proposition A.2]{Gen20}, take a geodesic ray $p$ in $X$ converging to $x$. Let $F \subset \lk_\G(x)$ be finite. By Lemma \ref{lem:geod rays of same Roller limit} (3), for any $y \in F$, there exist $n_y,m_y \in \NN$ and a geodesic ray $q^y=(q_0^y,q_1^y,\ldots)$ in $X$ converging to $y$ such that for any $i \in \NN$, we have $d_X(p_{n_y+i},q_{m_y+i}^y)=1$. Define $N \in \NN$ by $N = 1+\max_{y \in F} n_y$. Note $\{p_{N-1},p_{N+1}\} \subset \lk_\G(p_N)$. Hence, we have $|F|+2 \le |\lk_\G(p_N)| \le D$. Since $F$ is arbitrary, this implies $|\lk_\G(x)| \le D-2$ for any $x \in \partial_\R X$.
\end{proof}

In Lemma \ref{lem:from Roller bdry to Gromov bdry} below, $X$ is hyperbolic in $\ell^1$-metric and $\partial_\infty X$ is the Gromov boundary of $X$. 

\begin{lem}\label{lem:from Roller bdry to Gromov bdry}
    Let $X$ be a $\delta$-hyperbolic CAT(0) cube complex with $\delta\ge0$.
    \begin{itemize}
        \item[(1)]
        If $p$ and $q$ are geodesic rays in $X$ converging to the same point in $\partial_\R X$, then $p$ and $q$ converge to the same point in $\partial_\infty X$.
        \item[(2)]
        Define a map $f \colon \partial_\R X \to \partial_\infty X$ as follows; for $\xi \in \partial_\R X$, take a geodesic ray $p$ in $X$ converging to $\xi$ and define $f(\xi)$ to be the limit of $p$ in $\partial_\infty X$. Then, for any $x,y \in \partial_\R X$, (i)-(iii) are all equivalent; (i) $d_\G(x,y) \le \delta$, (ii) $d_\G(x,y) < \infty$, (iii) $f(x)=f(y)$.
    \end{itemize}
\end{lem}

\begin{proof}
    (1) Fix $o \in X^{(0)}$. Since there exists a geodesic ray $p'$ (resp. $q'$) from $o$ that merges $p$ (resp. $q$) (e.g. \cite[Lemma 3.27]{Oya24}), we may assume that both $p$ and $q$ are from $o$. Let $p=(p_1,p_2,\ldots)$ and $q=(q_1,q_2,\ldots)$ be the sequences of $0$-cubes composing $p$ and $q$. By Lemma \ref{lem:geod rays of same Roller limit} (1), for any $n \in \NN$, there exists $m \in \NN$ such that $d_X(o,p_n) + d_X(p_n,q_m) = d_X(o,q_m)$. Hence, $d_X(p_n,q_n) \le  \delta$ since $X$ is $\delta$-hyperbolic. Since $n$ is arbitrary, $p$ and $q$ converge to the same point in $\partial_\infty X$.

    (2)
    Note that the map $f \colon \partial_\R X \to \partial_\infty X$ is well-defined by \cite[Proposition A.2]{Gen20} (i.e. a geodesic ray $p$ converging to $\xi$ exists) and Lemma \ref{lem:from Roller bdry to Gromov bdry} (1). 
    
    $(i)\Rightarrow(ii)$ is trivial and $(ii)\Rightarrow(iii)$ follows from Lemma \ref{lem:geod rays of same Roller limit} (3). We'll show $(iii)\Rightarrow(i)$. Fix $o\in X^{(0)}$ and take geodesic rays $p=(p_0,p_1,\ldots)$ and $q=(q_0,q_1,\ldots)$ in $X$ from $o$ that converge to $x$ and $y$ respectively. Given a finite set $F \subset \H(x,y)$, there exists $n \in \NN$ such that $F \subset \H(p_n,q_n)$. By $f(x)=f(y)$, we have $d_X(p_n,q_n) \le \delta$. Hence, $|F| \le |\H(p_n,q_n)| = d_X(p_n,q_n) \le \delta$. Since $F$ is arbitrary, this implies $d_\G(x,y)=|\H(x,y)| \le \delta$ by Lemma \ref{lem:geod rays of same Roller limit} (2).
\end{proof}

In Lemma \ref{lem:unif. loc .fin. hyp.} below, we cannot weaken uniform local finiteness to local finiteness. Indeed, in the example of \cite[Remark 2.9]{Oya25}, the Gromov boundary is a singleton, but the Roller boundary contains infinitely many points. Note that the map $f$ being $R$-to-1 below means that for any $x \in \partial_\infty X$, $|f^{-1}(x)|\le R$.

\begin{lem}\label{lem:unif. loc .fin. hyp.}
    Let $X$ be a uniformly locally finite hyperbolic CAT(0) cube complex and let $f \colon \partial_\R X \to \partial_\infty X$ be the map as in Lemma \ref{lem:from Roller bdry to Gromov bdry} (2). Then, $f$ is $R$-to-1 with some $R \in \NN$, surjective, continuous, and $\aut(X)$-equivariant, where $\aut(X)$ is the automorphism group of $X$.
\end{lem}

\begin{proof}
    The map $f$ is $R$-to-1 with some $R \in \NN$ by Lemma \ref{lem:geod rays of same Roller limit} (4) and Lemma \ref{lem:from Roller bdry to Gromov bdry} (2). Surjectivity follows from local finiteness of $X$. Continuity can be shown in a similar way to the proof of Lemma \ref{lem:geod rays of same Roller limit} (1). $\aut(X)$-equivariance is straightforward.
\end{proof}

Proposition \ref{prop:hyperfinite of Roller and Gromov} below verifies that Theorem \ref{thm:main} is a generalization of \cite{HSS20}, by deriving Corollary \ref{cor:hyperfinite Roller bdry of cubulated hyp gp}.

\begin{prop}\label{prop:hyperfinite of Roller and Gromov}
    Let $X$ be a uniformly locally finite hyperbolic CAT(0) cube complex. Suppose that a countable group $G$ acts on $X$ cubically. Then, $E_G^{\partial_\R X}$ is hyperfinite if and only if $E_G^{\partial_\infty X}$ is hyperfinite.
\end{prop}

\begin{proof}
    This follows from Lemma \ref{lem:hyperfininte when finite to 1} and Lemma \ref{lem:unif. loc .fin. hyp.}.
\end{proof}

\begin{cor}[{\cite[Theorem 1.1]{HSS20}}]\label{cor:hyperfinite Roller bdry of cubulated hyp gp}
    If a hyperbolic group $G$ acts properly and cocompactly on a CAT(0) cube complex X, then the orbit equivalence relation $E_G^{\partial_\infty X}$ induced by the action $G \act \partial_\infty X$ is hyperfinite.
\end{cor}

\begin{proof}
    This follows from Corollary \ref{cor:cubulated hyp. gp} and Proposition \ref{prop:hyperfinite of Roller and Gromov}.
\end{proof}


\providecommand{\bysame}{\leavevmode\hbox to3em{\hrulefill}\thinspace}
\providecommand{\MR}{\relax\ifhmode\unskip\space\fi MR }
\providecommand{\MRhref}[2]{%
  \href{http://www.ams.org/mathscinet-getitem?mr=#1}{#2}
}
\providecommand{\href}[2]{#2}

\vspace{5mm}

\noindent  Department of Mathematics and Statistics, McGill University, Burnside Hall,

\noindent 805 Sherbrooke Street West, Montreal, QC, H3A 0B9, Canada.

\noindent E-mail: \emph{koichi.oyakawa@mail.mcgill.ca}

\end{document}